\documentclass[a4paper,11pt,twoside]{article}
\usepackage[english]{babel}
\usepackage[utf8]{inputenc}

\usepackage[a4paper,inner=3cm,outer=3cm,top=4cm,bottom=4cm,pdftex]{geometry}
\usepackage{fancyhdr}
\usepackage{sectsty}
\usepackage{titlesec}
\usepackage{bold-extra}
\titleformat{\section}{\normalfont\scshape\centering}{\thesection}{1em}{}
  \titleformat{\subsection}{\bfseries}{\thesubsection}{1em}{}
  
\usepackage{comment}
\usepackage{graphics}
\usepackage{aliascnt}
\usepackage[pdftex,citecolor=green,linkcolor=red]{hyperref}

\usepackage{amsmath}
\usepackage{amsfonts}
\usepackage{amssymb}
\usepackage{amsthm}
\newcommand{\qedd}{\hfill \ensuremath{\Box}}
\newtheorem{theorem}{Theorem}

\newtheorem{lemma}{Lemma}
\newtheorem{definition}{Definition}
\newtheorem{proposition}{Proposition}
\newtheorem{remark}{Remark}

\title{\scshape \Large Almost primes in almost all short intervals}
\author{\scshape \normalsize Joni Ter\"av\"ainen}
\date{}
\begin{document}
\maketitle

\begin{abstract}
Let $E_k$ be the set of positive integers having exactly $k$ prime factors. We show that almost all intervals $[x,x+\log^{1+\varepsilon} x]$ contain $E_3$ numbers, and almost all intervals $[x,x+\log^{3.51} x]$ contain $E_2$ numbers. By this we mean that there are only $o(X)$ integers $1\leq x\leq X$ for which the mentioned intervals do not contain such numbers. The result for $E_3$ numbers is optimal up to the $\varepsilon$ in the exponent. The theorem on $E_2$ numbers improves a result of Harman, which had the exponent $7+\varepsilon$ in place of $3.51$. We also consider general $E_k$ numbers, and find them on intervals whose lengths approach $\log x$ as $k\to \infty$.
\end{abstract}

\section{Introduction}

When studying $E_k$ numbers (products of exactly $k$ primes), it is natural to ask, how short intervals include such numbers almost always. Since Wolke's work \cite{wolke}, the essential question has been minimizing the number $c$ such that almost all intervals $[x,x+\log^{c}x]$ contain an $E_k$ number, meaning that all but $o(X)$ such intervals with integer $x\in [1,X]$ contain such a number. Wolke showed in 1979 that the value $c=5\cdot 10^6$ is admissible for $E_2$ numbers. This was improved to $c=7+\varepsilon$ for $E_2$ numbers by Harman \cite{harman-almostprimes} in 1982. Wolke's and Harman's methods are based on reducing the problem to estimates for sums over the zeros of the Riemann zeta function, and on the fact that the density hypothesis is known to hold in a non-trivial strip (namely Jutila's \cite{jutila-density} region $\sigma\geq \frac{11}{14}$ in Harman's argument\footnote{In fact, introducing into Harman's argument the widest known den'sity hypothesis region $\sigma\geq \frac{25}{32}$, due to Bourgain \cite{bourgain} from 2000, would give $c=6.86$.} ). To the author's knowledge, Harman's exponent for $E_2$ numbers was the best one known also for $E_k$ numbers with $k\geq 3$.\\

If one considers $P_k$ numbers, which are products of no more than $k$ primes, one can obtain improvements. Mikawa \cite{mikawa} showed in 1989 that for any function $\psi(x)$ tending to infinity, the interval $[x,x+\psi(x)\log^{5}x]$ contains a $P_2$ number almost always. Furthermore, Friedlander and Iwaniec \cite[Chapters 6 and 11]{friedlander} proved that for any such function $\psi(x)$ the interval $[x,x+\psi(x)\log x]$ contains a $P_4$ number almost always. They also hint how to prove the same result for $P_3$ numbers. There is however a crucial difference between $E_k$ and $P_k$ numbers, since the $E_k$ numbers are subject to the famous parity problem, and hence cannot be dealt with using only classical combinatorial sieves, which are the basis of the arguments on $P_k$ numbers. Therefore, the $E_k$ numbers are also a much closer analog of primes than the $P_k$ numbers.\\

One would naturally expect almost all intervals $[x,x+\psi(x)\log x]$ to have also prime numbers in them, and this would follow from the heuristic that the proportion of $x$ for which $[x,x+\lambda \log x]$ contains exactly $m$ primes for fixed $m$ and $\lambda>0$ should be given by the Poisson distribution $\frac{\lambda^m}{m!}e^{-\lambda}$. Such results are however far beyond the current knowledge, as the shortest intervals, almost all of which are known to contain primes, are $[x,x+x^{\frac{1}{20}+\varepsilon}]$ by a result of Jia \cite{jia}. However, the results of Goldston-Pintz-Y{\i}ld{\i}r{\i}m \cite{goldston1},\cite{goldston2} on short gaps between primes tell that for any $\lambda>0$ there is a positive proportion of integers $x\leq X$ for which $[x,x+ \lambda\log x]$ contains a prime, but it is not known whether this proportion approaches $1$ as $\lambda$ increases. A recent result of Freiberg \cite{freiberg}, in turn, gives exactly $m$ primes on an interval $[x,x+\lambda \log x]$ for at least $X^{1-o(1)}$ integers $x\leq X$. Concerning conditional results, Gallagher \cite{gallagher} showed that the Poisson distribution of primes in short intervals would follow from a certain uniform form of the Hardy-Littlewood prime $k$-tuple conjecture. Under the Riemann hypothesis, it was shown by Selberg \cite{selberg} in 1943 that almost all intervals $[x,x+\psi(x)\log^2 x]$ contain primes. For $E_2$ numbers, under the density hypothesis, Harman's argument from \cite{harman-almostprimes} would give the exponent $c=3+\varepsilon$.\\

In this paper, we establish the exponent $c=1+\varepsilon$ for $E_3$ numbers and the exponent $c=3.51$ for $E_2$ numbers. Our results for $E_2,E_3$ and $E_k$ numbers are stated as follows.

\begin{theorem}\label{t1} Almost all intervals $[x,x+(\log \log x)^{6+\varepsilon}\log x]$ contain a product of exactly three distinct primes.\end{theorem}

\begin{theorem} \label{t2} For any integer $k\geq 4$, there exists $C_k>0$ such that almost all intervals $[x,x+(\log_{k-1} x)^{C_k}\log x]$ contain a product of exactly $k$ distinct primes. Here $\log_{\ell}$ is the $\ell$th iterated logarithm.\end{theorem}

\begin{theorem} \label{t3}\label{3} Almost all intervals $[x,x+\log^{3.51} x]$ with $x\leq X$ contain a product of exactly two distinct primes.\end{theorem}

Theorems \ref{t1} and \ref{t2} are direct consequences of the following theorem.

\begin{theorem}\label{t4}
Let $X$ be large enough, $k\geq 3$ a fixed integer, and $\varepsilon>0$ small enough but fixed. Define the numbers $P_1,...,P_{k-1}$ by setting $P_{k-1}=(\log X)^{\varepsilon^{-2}},P_{k-2}=(\log \log X)^{6+10\sqrt{\varepsilon}}$ and $P_j=(\log P_{j+1})^{\varepsilon^{-1}}$ for $1\leq j\leq k-3$.
 For $P_1\log X\leq h \leq X$, we have
\begin{align}\label{eq15}
\left|\frac{1}{h}\sum_{\substack{x\leq p_1\dotsm p_k\leq x+h\\P_i\leq p_i\leq P_i^{1+\varepsilon},\,i\leq k-1}}1-\frac{1}{X}\sum_{\substack{X\leq p_1\dotsm p_k\leq 2X\\P_i\leq p_i\leq P_i^{1+\varepsilon},\,i\leq k-1}}1\right|\ll\frac{1}{(\log X)(\log_k X)}
\end{align} 
for almost all $x\leq X$.
\end{theorem}

In the theorem above, the average over the dyadic interval is $\gg \frac{1}{\log X}$ by the prime number theorem, so Theorems \ref{t1} and \ref{t2} indeed follow from Theorem \ref{t4}. Similarly, Theorem \ref{t3} is a direct consequence of the following.

\begin{theorem}\label{t5}
Let $X$ be large enough, $P_1=\log^a X$ with $a=2.51$, $\varepsilon>0$ fixed, and $P_1\log X\leq h\leq X$. We have 
\begin{align}\label{eq25}
\frac{1}{h}\sum_{\substack{x\leq p_1p_2\leq x+h\\ P_1\leq p_1<P_1^{1+\varepsilon}}}1\gg \frac{1}{X}\sum_{\substack{X\leq p_1p_2\leq 2X\\ P_1\leq p_1\leq P_1^{1+\varepsilon}}}1
\end{align}
for almost all $x\leq X$. 
\end{theorem}

\begin{remark}
Since $h\geq P_1\log X$, we have the dependence $c=a+1$ between the exponent $a$ in Theorem \ref{t5} and the smallest exponent $c$ for which we can show that the interval $[x,x+\log^c x]$ contains an $E_2$ number almost always. 
\end{remark}

\begin{remark}
Note that Theorems \ref{t4} and \ref{t5} tell us that there are $\gg \frac{h}{\log X}$ $E_k$ numbers in almost all intervals $[x,x+h]$, where $h$ and $k$ are as in one of the theorems. However, we are not quite able to find $E_k$ numbers on intervals $[x,x+\psi(x)\log x]$ with $\psi$ tending to infinity arbitrarily slowly, unlike in the result of Friedlander and Iwaniec on $P_k$ numbers. In addition, our bound for the number of exceptional values is at best $\ll \frac{x}{\log^{\varepsilon}x}$ and often weaker, while the methods used in \cite{harman-sieves}, \cite{jia} and \cite{watt-primes} for primes in almost all short intervals have a tendency to give the bound $\ll \frac{x}{\log^{A}x}$ for any $A>0$, when they work. The limit of our method for $E_2$ numbers is the exponent $3+\varepsilon$, as will be seen later, so proving for example unconditionally the analog of Selberg's result for $E_2$ numbers would require some further ideas.
\end{remark}

To prove our results, we adapt the ideas of the paper \cite{matomaki} of Matom\"aki and Radziwi\l{}\l{} on multiplicative functions in short intervals to considering almost primes in short intervals. In that paper, a groundbreaking result is that for any multiplicative function, with values in $[-1,1]$, its average over $[x,x+h]$ is almost always asymptotically equal to its dyadic average over $[x,2x]$, with $h=h(x)\leq x$ any function tending to infinity. The error terms obtained there for general multiplicative functions are not quite good enough for our purposes. Nevertheless, using similar techniques, and replacing the multiplicative function with the indicator function of the numbers $p_1\dotsm p_k$, with $p_i$ primes from carefully chosen intervals, allows us to find $E_k$ numbers on intervals $[x,x+h]$, with $\frac{h}{\log x}$ growing very slowly. In this setting, we can apply various mean, large and pointwise value results for Dirichlet polynomials, some of which work specifically with primes or the zeta function, but not with general multiplicative functions (such as Watt's theorem on the twisted moment of the Riemann zeta function, a large values theorem from \cite{matomaki} for Dirichlet polynomials supported on primes, and Vinogradov's zero-free region). In many places in the argument, we cannot afford to lose even factors of $\log^{\varepsilon}x$, so we need to factorize Dirichlet polynomials in a manner that is nearly nearly lossless, and use an improved form of the mean value theorem for Dirichlet polynomials. To deal with some of the arising Dirichlet polynomials, we also need some sieve methods, similar to those that have been successfully applied to finding primes in short intervals for example in \cite{harman-sieves}, \cite{jia} and \cite{watt-primes}. In the case of $E_2$ numbers, in addition to these methods, we benefit from the theory of exponent pairs and Jutila's large values theorem.\\

The structure of the proofs of Theorems \ref{t4} and \ref{t5} is the following. We will first present the lemmas necessary for proving Theorem \ref{t4}, and hence Theorems \ref{t1} and \ref{t2}. Besides employing these lemmas to prove Theorem \ref{t4}, we notice that they are already sufficient for finding products of exactly two primes in almost all intervals $[x,x+\log^{5+\varepsilon} x]$, which is as good as Mikawa's result for $P_2$ numbers up to $\varepsilon$ in the exponent (one could also get $c$ slightly below $5$ using exponent pairs, which are just one of the additional ideas required for Theorem \ref{t5}). The rest of the paper is then concerned with reducing the exponent $5+\varepsilon$ to $3.51$ for products of two primes, and this requires some further ingredients, as well as all the lemmas that were needed for products of three or more primes.

\subsection{Acknowledgements}

The author is grateful to his supervisor Kaisa Matom\"aki for various useful comments and discussions. The author thanks the referee for careful reading of the paper and for useful comments. While working on this project, the author was supported by the Vilho, Yrj\"o and Kalle Vais\"al\"a foundation of the Finnish Academy of Science and Letters.

\subsection{Notation}\label{subsec:notation}

The symbols $p,q,p_i$ and $q_i$ are reserved for primes, and $d,k,\ell, m$ and $n$ are always positive integers. We often use the same capital letter for a Dirichlet polynomial and its length. We call \textit{zeta sums} partial sums of $\zeta(s)$ or $\zeta'(s)$ of the form $\sum_{n\sim N}n^{-s}$ or $\sum_{n\sim N}(\log n)n^{-s}$. \\

The function $\nu(\cdot)$ counts the number of distinct prime divisors of a number, $\mu(\cdot)$ is the Möbius function, $\Lambda(\cdot)$ is the von Mangoldt function, and $d_r(m)$ is the number of solutions to $a_1\dotsm a_r=m$ in positive integers. The function $\omega(\cdot)$ is Buchstab's function (see Harman's book \cite[Chapter 1]{harman-sieves}), defined as $\omega(u)=\frac{1}{u}$ for $1\leq u\leq 2$ and via the differential equation $\frac{d}{du}(u\omega(u))=\omega(u-1)$ for $u>2$, imposing the requirement that $\omega$ be continuous on $[1,\infty)$. We make the convention that $\omega(u)=0$ for $u<1$. In addition, let $\mathcal{P}(z)=\prod_{p<z}p$, and let $S(A,\mathbb{P},z)$ count the numbers in $A$ coprime to $\mathcal{P}(z)$.\\

The quantity $\varepsilon>0$ is always small enough but fixed. The symbols $C_1,C_2,...$ denote unspecified, positive, absolute constants. By writing $n\sim X$ in a summation, we mean $X\leq n<2X$. The expression $1_S$ is the indicator function of the set $S$, so that $1_S(n)=1$ if $n\in S$ and $1_S(n)=0$ otherwise. We use the usual Landau and Vinogradov asymptotic notation $o(\cdot), O(\cdot)$ and $\ll, \gg$. The notation $X\asymp Y$ is shorthand for $X\ll Y\ll X$.

\section{Preliminary lemmas}

\subsection{Reduction to mean values of Dirichlet polynomials}\label{subsec:reduction}

We present several lemmas that are required for proving both Theorems \ref{t4} and \ref{t5}. Later on, we give some additional lemmas that are needed only for proving Theorem \ref{t5}.\\

The plan of the proofs of Theorems \ref{t4} and \ref{t5}, and hence of Theorems \ref{t1}, \ref{t2} and \ref{t3}, is to transform the problem of comparing almost primes in short and long intervals to finding cancellation in the mean square of the corresponding Dirichlet polynomial. The polynomial can be factorized after it is divided into short intervals, and different methods can be applied to different factors. This approach is utilized in many earlier works on primes and almost primes in short intervals; see e.g. \cite{harman-sieves}, \cite{matomaki}. We then apply several mean, large and pointwise value theorems, which are presented in Subsection \ref{subsec:Dirichlet bounds}, to find the desired cancellation in the Dirichlet polynomial.\\

The following Parseval-type lemma allows us to reduce the problem of finding almost primes in short intervals to finding cancellation in a Dirichlet polynomial.

\begin{lemma}\label{1} Let
\begin{align*}
S_h(x)=\frac{1}{h}\sum_{x\leq n\leq x+h}a_n,
\end{align*}
where $a_n$ are complex numbers, and let $2\leq h_1\leq h_2\leq \frac{X}{T_0^3}$ with $T_0\geq 1$. Also let $F(s)=\sum_{n\sim X}\frac{a_n}{n^s}$. Then 
\begin{align}\label{eq17}
\frac{1}{X}\int_{X}^{2X}\left|\frac{1}{h_1}S_{h_1}(x)-\frac{1}{h_2}S_{h_2}(x)\right|^2 dx&\ll \frac{1}{T_0}+\int_{T_0}^{\frac{X}{h_1}}|F(1+it)|^2 dt\nonumber\\
&+\max_{T\geq \frac{X}{h_1}}\frac{X}{Th_1}\int_{T}^{2T}|F(1+it)|^2 dt.
\end{align}
\end{lemma}

\begin{proof} This is Lemma 14 in the paper \cite{matomaki} (except that we do not specify the value of $T_0$). A related bound can be found for example in \cite[Chapter 9]{harman-sieves}.\end{proof}

We choose $T_0=X^{0.01}$, and $h_2=\frac{X}{T_0^3}$ in Lemma \ref{1}, and the average function $S_h(x)$ is given by the short average in \eqref{eq15} or \eqref{eq25}. Now, defining
\begin{align*}
F(s)=\sum_{\substack{p_1\dotsm p_k\sim X\\P_i\leq p_i\leq  P_i^{1+\varepsilon},i\leq k-1}}(p_1\dotsm p_k)^{-s},
\end{align*}
where $P_i$ are as in Theorem \ref{t4} or \ref{t5}, proving Theorems \ref{t4} and \ref{t5} is reduced to showing that
\begin{align}\label{eq24}
\int_{T_0}^{T}|F(1+it)|^2 dt=o\left(\left(\frac{Th}{X}+1\right)\cdot \frac{1}{(\log^2 X)(\log_ {k} X)^{2}}\right),
\end{align}
for $T_0=X^{0.01}$ and $h\geq P_1\log X$. Indeed, substituting this to Lemma \ref{1} shows that 
\begin{align*}
\frac{1}{X}\int_{X}^{2X}\left|\frac{1}{h}S_{h}(x)-\frac{1}{h_2}S_{h_2}(x)\right|^2dx=o\left(\frac{1}{(\log^2 X)(\log _k X)^{2}}\right),
\end{align*}
where $h_2=\frac{X}{T_0^3}$. It actually suffices to prove \eqref{eq24} for $T\leq X$, since otherwise the mean value theorem (Lemma \ref{12}) gives a good enough bound for the last term in \eqref{eq17}.\\

Note that for $T\leq X$ the trivial bound for the integral in \eqref{eq24}, coming from the mean value theorem, is $\ll (\log X)^{-1}$. Thus our task is to save slightly more than one additional logarithm in this integral (for $T\leq \frac{X}{h}$, at least). \\

Once the required estimates for Dirichlet polynomials have been established, we can apply the prime number theorem in short intervals with Vinogradov's error term (see \cite[Chapter 10]{iwaniec-kowalski}) to see that
\begin{align*}
\frac{1}{h_2}S_{h_2}(x)-\frac{1}{X}S_X(X)\ll \exp(-(\log
X)^{\frac{3}{5}-\varepsilon}),
\end{align*}
for $h_2=x^{0.97},x\sim  X$, and hence deduce Theorems \ref{t4} and \ref{t5} (and consequently \ref{t1}, \ref{t2} and \ref{t3}). For example, we compute
\begin{align*}
\frac{1}{h_2}\sum_{\substack{x\leq p_1p_2p_3\leq x+h_2\\P_1\leq p_1\leq P_1^{1+\varepsilon}\\P_2\leq p_2\leq P_2^{1+\varepsilon}}}1&=\frac{1}{h_2}\sum_{\substack{P_1\leq p_1\leq P_1^{1+\varepsilon}\\P_2\leq p_2\leq P_2^{1+\varepsilon}}}\left(\pi\left(\frac{x+h_2}{p_1p_2}\right)-\pi\left(\frac{x}{p_1p_2}\right)\right)\\
&=\frac{1}{h_2}\sum_{\substack{P_1\leq p_1\leq P_1^{1+\varepsilon}\\P_2\leq p_2\leq P_2^{1+\varepsilon}}}\frac{h_2}{p_1p_2\log \frac{x}{p_1p_2}}\\
&\quad+O\left(\exp(-(\log x)^{\frac{3}{5}-\frac{\varepsilon}{2}})\right)\\
&=\sum_{P_1\leq p_1\leq P_1^{1+\varepsilon}\atop P_2\leq p_2\leq P_2^{1+\varepsilon}}\frac{1}{p_1p_2\log \frac{X}{p_1p_2}}+O(\exp(-(\log X)^{\frac{3}{5}-\varepsilon})),
\end{align*}
and the same asymptotics hold for the dyadic sum. Sometimes we end up comparing the sums $\frac{1}{h_2}S_{h_2}(x)$ and $\frac{1}{x}S_2(x)$ with $a_n$ not quite equal to the coefficients of $F(s)$, but equal to the indicator function of the numbers $p_1p_2n$ with $p_1$ and $p_2$ from the intervals $[P_1,P_1^{1+\varepsilon}]$ and $[P_2,P_2^{1+\varepsilon}],$ respectively, and $n$ having no prime factors smaller than $p_2$. There may also be a simple cross-conditions on $p_1$ and $p_2$, but comparing the sums still causes no difficulty.\\

Thus, in the rest of the paper we can concentrate on bounding Dirichlet polynomials. Although there is a close analogy in the formulations of Theorems \ref{t4} and \ref{t5}, estimating the polynomial arising from the latter is more difficult, and will require several additional ideas.

\subsection{Factorizations for Dirichlet polynomials}

In bounding Dirichlet polynomials, factorizations play an important role. We encounter situations where the only cross-condition on the variables in the polynomial is that their product belongs to a certain range, so the variables can be separated by diving them into short ranges and estimating the mean values of the resulting polynomials. The factorization is provided by the following lemma, which also takes into account the improved mean value theorem (Lemma \ref{2}).

\begin{lemma} \label{6}Let $\mathcal{S}\subset [-T,T]$ be measurable and 
\begin{align*}
F(s)=\sum_{\substack{mn\sim X\\M\leq m\leq M'}}\frac{a_mb_n}{(mn)^s}
\end{align*}
for some $M'>M\geq 2$ and for some complex numbers $a_m,b_n$. Let $H\geq 1$ be such that $H\log M$ and $H\log M'$ are integers. Denote
\begin{align*}
A_{v,H}(s)=\sum_{e^{\frac{v}{H}}\leq m<e^{\frac{v+1}{H}}}\frac{a_m}{m^s},\quad B_{v,H}(s)=\sum_{n\sim Xe^{-\frac{v}{H}}}\frac{b_n}{n^s}.
\end{align*}
Then
\begin{align*}
\int_{\mathcal{S}}|F(1+it)|^2dt&\ll |I|^2\int_{\mathcal{S}}|A_{v_0,H}(1+it)B_{v_0,H}(1+it)|^2dt\\
&+T\sum_{\substack{n\in [Xe^{-\frac{1}{H}},Xe^{\frac{1}{H}}]\,\, \text{or}\\ n\in [2X,2Xe^{\frac{1}{H}}]}}|c_n|^2+T\sum_{1\leq h\leq \frac{2X}{T}}\sum_{\substack{m-n=h\atop m,n\in [Xe^{-\frac{1}{H}},Xe^{\frac{1}{H}}]\,\, \text{or}\\ m,n\in [2X,2Xe^{\frac{1}{H}}]\\}}|c_m||c_n|,
\end{align*}
with
\begin{align*}
c_n=\frac{1}{n}\sum_{n=k\ell\atop M\leq k\leq M'}|a_kb_{\ell}|,
\end{align*}
$I=[H\log M,H\log M')$ and $v_0\in I$ a suitable integer.\end{lemma} 
\begin{remark}
In applications we have $M'\geq 2M$, so the conditions that $H\log M$ and $H\log M'$ be integers can be ignored, since we can always afford to vary $H$ and $M'$ by the necessary amount.
\end{remark}

\begin{remark}
When proving Theorem \ref{t4}, we cannot afford to lose any powers of logarithm in some factorizations, and indeed the second term in the lemma crucially has the factor $T$ instead of the factor $X$ occurring in the mean value theorem, and in the first term we will lose a factor of size $\ll H^2\log^2 \frac{M'}{M}$, which in practice is minuscule. 
\end{remark}

\begin{proof} This resembles Lemma 12 in the paper \cite{matomaki} by Matomäki and Radziwi\l{}\l{} (where, in addition to factorization in short intervals, a Ramaré-type identity is used). We split $F(s)$ into short intervals, obtaining
\begin{align*}
F(s)=\sum_{v\in I\cap \mathbb{Z}}\sum_{e^{\frac{v}{H}}\leq m<e^{\frac{v+1}{H}}}\frac{a_m}{m^s}\sum_{Xe^{-\frac{v+1}{H}}\leq n<2Xe^{-\frac{v}{H}}\atop mn\sim X}\frac{b_n}{n^s}.
\end{align*}
Observe that $Xe^{-\frac{v+1}{H}}\leq n<Xe^{-\frac{v}{H}}$ can hold above only for $mn\in [Xe^{-\frac{1}{H}},Xe^{\frac{1}{H}}]$. Furthermore, we always have $mn\in [Xe^{-\frac{1}{H}},2Xe^{\frac{1}{H}}]$. This allows us to write
\begin{align}\label{eq43}
F(s)=\sum_{v\in I\cap \mathbb{Z}}A_{v,H}(s)B_{v,H}(s)+\sum_{k\in [Xe^{-\frac{1}{H}},Xe^{\frac{1}{H}}]\,or\atop k\in [2X,2Xe^{\frac{1}{H}}]}\frac{d_{k}}{k^s}
\end{align}
with
\begin{align*}
|d_{k}|\leq \sum_{k=mn}|a_mb_{n}|.
\end{align*}
Now the claim of the lemma follows by taking mean squares on both sides of \eqref{eq43} on the line $\Re(s)=1$, applying the improved mean value theorem (Lemma \ref{2}), and taking the maximum in the sum over $I$.\end{proof}

\subsection{Bounds for Dirichlet polynomials}\label{subsec:Dirichlet bounds}

We need several mean, large and pointwise value results on Dirichlet polynomials. The following lemma is one of the basic tools.

\begin{lemma}{(Mean value theorem for Dirichlet polynomials)}\label{12} Let $N\geq 1$ and $F(s)=\sum_{n\sim N}\frac{a_n}{n^s}$, where $a_n$ are any complex numbers. Then
\begin{align*}
\int_{-T}^{T}|F(it)|^2 dt\ll (N+T)\sum_{n\sim N}|a_n|^2.
\end{align*}\end{lemma}

\begin{proof} See for example Iwaniec and Kowalski's book \cite[Chapter 9]{iwaniec-kowalski}.\end{proof}

If the coefficients $a_n$ are supported on the primes or almost primes and are of size $\asymp \frac{1}{n}$, the sum $\sum_{n\sim N}|a_n|^2$ is essentially $\frac{1}{N\log N}$. However, in some places in the proofs of Theorems \ref{t1}, \ref{t2} and \ref{t3}, it is vital to save one more logarithm in such a situation. This is enabled by an improved mean value theorem.

\begin{lemma}{(Improved mean value theorem)}\label{2} Let $N$ and $F(s)$ be as above. We have
\begin{align}\label{eq5}
\int_{-T}^{T}|F(it)|^2dt\ll T\sum_{n\sim N}|a_n|^2+T\sum_{1\leq h\leq \frac{N}{T}}\sum_{m-n=h\atop m,n\sim N}|a_m||a_n|.
\end{align}\end{lemma}

\begin{remark}
The number of solutions to $m-n=h$, with $m$ and $n$ primes and $m,n\sim N$, is $\ll \frac{N^2}{\log^2 N}\cdot \frac{h}{\varphi(h)}$ (with $\varphi$ Euler's totient function), which follows easily from Brun's sieve, for example. If $T\leq \frac{N}{h}, h\geq \log N$ and $a_n$ is supported on the primes, the first sum in \eqref{eq5} turns out not to be problematic, so we indeed save essentially one additional logarithm with this lemma. We remark that if we have polynomials of length $N\leq T$, Lemma \ref{2} reduces to the basic mean value theorem.
\end{remark}

\begin{proof} This follows from Lemma 7.1 in \cite[Chapter 7]{iwaniec-kowalski}, taking $Y=10T$ there. \end{proof}

We also put into use a discrete mean value theorem, which is particularly useful when we take the mean square over a rather small set of points.

\begin{lemma}{(Halász-Montgomery inequality)}\label{18} Let $N$ and $F(s)$ be as before. Let $\mathcal{T}\subset [-T,T]$ be \textnormal{well-spaced}, meaning that $t,u\in \mathcal{T}$ and $t\neq u$ imply $|t-u|\geq 1.$ Then
\begin{align*}
\sum_{t\in \mathcal{T}}|F(it)|^2\ll (N+|\mathcal{T}|T^{\frac{1}{2}})(\log T)\sum_{n\sim N}|a_n|^2.
\end{align*}\end{lemma}

\begin{proof} For a proof, see Iwaniec and Kowalski's book \cite[Chapter 9]{iwaniec-kowalski}. \end{proof}

In addition to mean value theorems, we need some large values theorems. We come across some very short Dirichlet polynomials, say of length $\ll T^{o(1)}$, and we make use of the fact that the coefficients of these polynomials are supported on the primes.

\begin{lemma}\label{7} Let $P\geq 1,\,V>0$ and 
\begin{align*}
F(s)=\sum_{p\sim P}\frac{a_p}{p^s}
\end{align*}
with $|a_p|\leq 1$. Let $\mathcal{T}\subset [-T,T]$ be a well-spaced set of points such that $|F(1+it)|\geq V$ for each $t\in \mathcal{T}$. Then we have
\begin{align*}
|\mathcal{T}|\ll T^{2\frac{\log V^{-1}}{\log P}}V^{-2}\exp\left((1+o(1))\frac{\log T}{\log P}\log \log T\right).
\end{align*}\end{lemma}
\begin{remark}
We may also apply this lemma to polynomials not supported on primes, provided that $P\gg X^{\varepsilon}$ for some $\varepsilon>0$. In this case, the lemma is essentially the mean value theorem applied to a suitable moment of the polynomial.
\end{remark}

\begin{proof} This is Lemma 8 in the paper \cite{matomaki}. There a factor of $2$ occurs instead of $1+o(1)$ in the last exponential, but the exact same proof works with the factor $1+o(1)$.\end{proof} 

For proving Theorem \ref{t5}, we also need a large values theorem designed for long polynomials. The reason for presenting it along with the lemmas for Theorem \ref{t4} is that combining it with the other lemmas already gives the exponent $c=5+\varepsilon$ for $E_2$ numbers. The large values result is a theorem of Jutila that improves on the better known Huxley's large values theorem.

\begin{lemma}\label{13} (Jutila's large values theorem). Let $F(s)=\sum_{n\sim N}\frac{a_n}{n^s}$ with $|a_n|\leq d_r(n)$ for some fixed $r$. Let $\mathcal{T}\subset [-T,T]$ be a well-spaced set such that $|F(1+it)|\geq V$ for $t\in \mathcal{T}$, and let $k$ be any positive integer. We have
\begin{align*}
|\mathcal{T}|\ll \left(V^{-2}+\frac{T}{N^{2}}V^{-6+\frac{2}{k}}+V^{-8k}\frac{T}{N^{2k}}\right)(NT)^{o(1)}.
\end{align*}
\end{lemma}

\begin{proof} The proof can be found in Jutila's paper \cite{jutila-density}. We apply formula (1.4) there to $F(s)^{\ell}$, and have $G=\sum_{n\sim N}\frac{|a_n|^2}{n^2}\ll (NT)^{o(1)}N^{-1}$ in the notation of that paper.\end{proof}

In some cases in the proof of Theorem \ref{t4}, there will be polynomials supported on primes or almost primes for which the best we can do is apply a pointwise bound. These bounds follow in the end from Vinogradov's zero-free region. 

\begin{lemma}\label{17} Let 
\begin{align*}
P(s)=\sum_{n_1\dotsm n_k\sim N}g_1(n_1)\dotsm g_k(n_k)(n_1\dotsm n_k)^{-s},
\end{align*}
where $k\geq 1$ is a fixed integer and each $g_i$ is either the Möbius function, the characteristic function of the primes, the identity function, or the logarithm function. We have
\begin{align*}
|P(1+it)|\ll \exp\left(-(\log N)^{\frac{1}{10}}\right)
\end{align*}
when $\exp((\log N)^{\frac{1}{3}})\leq |t|\leq N^{A\log \log N}$ for any fixed $A>0$.\end{lemma}

\begin{proof} For $k=1$, the claim follows directly from Perron's formula and Vinogradov's zero-free region, so let $k\geq 2$. We may assume that $n_1,...,n_k$ belong to some dyadic intervals $I_1,...,I_k$ such that $I_k=[a,b]$ with $a\gg N^{\frac{1}{k}},b\ll N$. Now 
\begin{align*}
&\sum_{n_1\in I_1,...,n_{k-1}\in I_{k-1}}g(n_1)\dotsm g(n_{k-1})(n_1\dotsm n_{k-1})^{-1-it}\sum_{n_k\in I_k\atop  n_k\sim \frac{N}{n_1\dotsm n_{k-1}}}g(n_k) n_k^{-1-it}\\
&\ll (\log N)^{O(1)}\sum_{n_1\in I_1,...,n_{k-1}\in I_{k-1}}(n_1\dotsm n_{k-1})^{-1}\cdot \exp\left(-\frac{\log N^{\frac{1}{k}}}{(\log t)^{\frac{2}{3}+\varepsilon}}\right)\\
&\ll\exp\left(-(\log N)^{\frac{1}{10}}\right),
\end{align*}
as wanted.\end{proof}

\subsection{Moments of Dirichlet polynomials}

We need Watt's result on the twisted fourth moment of zeta sums (see Subsection \ref{subsec:notation} for the definition of zeta sums). This bound comes into play when we estimate the mean square of a product of Dirichlet polynomials where one of the polynomials is a long zeta sum.

\begin{lemma}\label{16} (Watt). Let $T\geq T_0\geq T^{\varepsilon}, T^{1+o(1)}\gg M,N\geq 1$. Define the Dirichlet polynomials  $N(s)=\sum_{n\sim N}n^{-s}$ or $N(s)=\sum_{n\sim N}(\log n)n^{-s}$ and $M(s)=\sum_{m\sim M}\frac{a_m}{m^s}$ with $a_m$ any complex numbers. We have
\begin{align*}
\int_{T_0}^{T}|N(1+it)|^4 |M(1+it)|^2dt\ll \left(\frac{T}{MN^2}(1+M^2T^{-\frac{1}{2}})+\frac{1}{T_0^3}\right)T^{o(1)}\max_{m\sim M}|a_m|^2.
\end{align*}
\end{lemma}

\begin{proof} An easy partial summation argument shows that we may assume $N(s)=\sum_{n\sim N}n^{-s}$. The lemma will be reduced to Watt's original twisted moment result \cite{watt-thm}, where $N(s)$ is replaced with $\zeta(s)$. It is well-known that $|N(1+it)|\ll \frac{1}{t}$ for $N\geq t\geq 1$ (see \cite[Chapter 8]{iwaniec-kowalski}), so
\begin{align*}
\int_{T_0}^{N}|N(1+it)|^4|M(1+it)|^2dt&\ll \max_{m\sim M}|a_m|^2\int_{T_0}^{T}\frac{1}{t^4}dt\cdot T^{o(1)}\\
&\ll \frac{T^{o(1)}}{T_0^3}\max_{m\sim M}|a_m|^2.
\end{align*}
Now it suffices to consider the integrals over dyadic intervals $[U,2U]$ with $N\leq U\leq T.$ These are bounded as in Lemma 2 of \cite{baker-harman-pintz} (using Watt's result and simple considerations), since translating the results there from the line $\Re(s)=\frac{1}{2}$ to the line $\Re(s)=1$ is an easy matter (and the bound in \cite{baker-harman-pintz} should be multiplied by $\max_{m\sim M}|a_m|^2$, as we do not assume $|a_m|\leq 1$).\end{proof}

\subsection{Sieve estimates}\label{subsec:sieve}

There are occasions in the proofs of Theorems \ref{t4} and \ref{t5} where our Dirichlet polynomials are too long, and we need a device for splitting them into shorter ones. This is enabled by Heath-Brown's identity and the decomposition resulting from it, which tells that either our Dirichlet polynomial can be replaced with a product of many polynomials, which is desirable, or it can be replaced with products of zeta sums, in which case we can make use of Watt's theorem.

\begin{definition}
A Dirichlet polynomial $M(s)=\sum_{m\sim M}\frac{a_m}{m^s}$ with $|a_n|\ll d_r(n)$ for fixed $r$ is called \textnormal{prime-factored} if, for each $A>0$, we have $|M(1+it)|\ll_A (\log M)^{-A}$ for $\exp((\log M)^{\frac{1}{3}})\leq t\leq M^{A\log \log M}$.
\end{definition}

\begin{lemma} (Heath-Brown's decomposition) \label{11} Let an integer $k\geq 1$ and a real number $\delta>0$ be fixed, and let $T\geq 2$. Define $P(s)=\sum_{P\leq p<P'}p^{-s}$ with $P\gg T^{\delta},P'\in \left[P+\frac{P}{\log T},2P\right]$. There exist Dirichlet polynomials $G_1(s),...,G_{L}(s)$ and a constant $C>0$ such that
\begin{align*}
|P(1+it)|\ll (\log^C X)(|G_1(1+it)|+\dotsm +|G_{L}(1+it)|)\quad \text{for all}\quad t\in [-T,T],
\end{align*}
with $L\leq \log^C X$, each $G_j(s)$ being of the form
\begin{align*}
G_j(s)=\prod_{i\leq J_j}M_i(s),\quad J_j\leq 2k,
\end{align*}
with $M_i(s)$ prime-factored Dirichlet polynomials (which depend on $j$), whose lengths satisfy $M_1\dotsm M_J=X^{1+o(1)},M_i\gg \exp\left(\frac{\log P}{\log \log P}\right)$. Additionally, each $M_i(s)$ with $M_i>X^{\frac{1}{k}}$ is a zeta sum.
\end{lemma}

\begin{proof} For a similar bound, see Harman's book \cite[Chapter 7]{harman-sieves}. It suffices to prove an analogous result for the polynomial $\sum_{P\leq n<P'}\Lambda(n)n^{-s}$ and use summation by parts. We take $f(n)=n^{-1-it}1_{[P,P']}(n)$ in the general Heath-Brown identity \cite{heath-brown-vaughan} for $\sum_{n\leq N}f(n)\Lambda(n)$, splitting each resulting variables into dyadic intervals, and  separating the variables with Perron's formula. The summation condition in Heath-Brown's identity guarantees that of the arising polynomials only the zeta sums can have length $>X^{\frac{1}{k}}$. If there are any polynomials of length $\ll \exp\left(\frac{\log P}{\log \log P}\right)$, these can simply be estimated trivially. The fact that the remaining polynomials of length $\gg \exp\left(\frac{\log P}{\log \log P}\right)$ are prime-factored follows from the fact that they have as their coefficients one of the sequences $(1), (\log n)$ and $(\mu(n))$, so that Lemma \ref{17} gives a pointwise saving of $\ll_A (\log P)^{-A}$. \end{proof}

There is one more lemma that we need on the coefficients of Dirichlet polynomials arising from almost primes. We need to bound the following quantities that are related to the quantities occurring in the improved mean value theorem for Dirichlet polynomials.

\begin{definition} For any sequence $(a_n)$ of complex numbers, set $X_1=\exp(\frac{\log X}{(\log \log X)^4})$ and
\begin{align*}
S_1(X,(a_n))&=\max_{\frac{X}{X_1}\leq Y\leq 4X\atop 1\leq H\leq \log^{10} X}H\sum_{Y\leq n\leq Y+\frac{Y}{H}}\frac{|a_n|^2}{n},\\
S_2(X,(a_n))&=\max_{\frac{X}{X_1}\leq Y\leq 4X\atop 1\leq H\leq \log^{10}X}H\sum_{1\leq h\leq \frac{X}{T}}\sum_{Y\leq n\leq Y+\frac{Y}{H}}\frac{|a_n||a_{n+h}|}{n}.
\end{align*}
\end{definition}

We get bounds of size essentially $\frac{1}{\log X}$ and $\frac{X}{T\log^2 X}$ for $S_1(X,(a_n))$ and $S_2(X,(a_n))$, respectively, under the assumptions of the next lemma.

\begin{lemma}\label{15}
Let $Z_r\geq \dotsm \geq Z_1\geq 1$ for a fixed $r$ with $Z_r\geq \exp(\frac{\log X}{(\log \log X)^3})$, $Z_r\leq z\leq 4X$, and
\begin{align*}
\mathcal{Q}=\left\{n\leq 4X:n=p_1\dotsm p_rm,\,\,p_i\in [Z_i,Z_i^2],\,\,(m,\mathcal{P}(z))=1)\right\}.
\end{align*}
Let $|a_n|\leq 1_{\mathcal{Q}}(n)$, and let $S_1(X,(a_n))$ and $S_2(X,(a_n))$ be as defined above. Then
\begin{align*}
S_1(X,(a_n))\ll \frac{1}{\log z}\quad \text{and}\quad S_2(X,(a_n))\ll \frac{1}{\log^2 z}\cdot \frac{X}{T}.
\end{align*}
\end{lemma}

\begin{remark}
Notice that we could also take as the set $\mathcal{Q}$ the set 
\begin{align*}
\mathcal{Q}'=\left\{n\leq 4X:n=p_1\dotsm p_rm,\,\,p_i\in [Z_i,Z_i^2],\,\,(m,\mathcal{P}(p_r))=1\right\}
\end{align*}
or the set
\begin{align*}
\mathcal{Q}''=\left\{n\leq 4X:n=p_1\dotsm p_r,\,\,p_i\in [Z_i,Z_i^2]\right\}.
\end{align*}
Indeed, the sizes of $\mathcal{Q}'$ and $\mathcal{Q}''$ can be bounded by sizes of sets of the form given in the lemma (with the parameter $z=Z_r$ or $z=X^{\frac{1}{r-1}}$). This observation will be used subsequently.
\end{remark}

\begin{proof} Let $S(A,\mathbb{P},z)$ count the numbers in $A$ having no prime factors below $z$, and let $\Pi$ be the product of all primes in $\bigcup_{i=1}^{r}[Z_i,Z_i^2]\cap[1,z]$. Brun's sieve yields
\begin{align*}
S_1(X,(a_n))&\ll \max_{\frac{X}{X_1}\leq Y\leq 4X\atop 1\leq H\leq \log^{10}X} \frac{H}{Y}\cdot \left|\left[Y,Y+\frac{Y}{H}\right]\cap \mathcal{Q}\right|\\
&\ll \max_{\frac{X}{X_1}\leq Y\leq 4X\atop 1\leq H\leq \log^{10}X} \frac{H}{Y}\cdot \left|\left\{n\in \left[Y,Y+\frac{Y}{H}\right]:\,\left(n,\frac{\mathcal{P}(z)}{\Pi}\right)=1\right\}\right|\\
&\ll \max_{\frac{X}{X_1}\leq Y\leq 4X\atop 1\leq H\leq \log^{10}X}\frac{H}{Y}\cdot\left(\frac{Y}{H\log z}+z^{\frac{1}{2}}\right)\\
&\ll \frac{1}{\log z},
\end{align*}
since $z^{\frac{1}{2}}\leq (4X)^{\frac{1}{2}}\ll \frac{Y}{H\log^2 z}$.\\

Furthermore, Brun's sieve also yields
\begin{align*}
S_2(X,(a_n))&\ll \max_{X_1\leq Y\leq 4X\atop 1\leq H\leq \log^{10}X}\frac{H}{Y}\sum_{1\leq h\leq \frac{X}{T}} \left|\left\{n\in \left[Y,Y+\frac{Y}{H}\right]:\left(n(n+h),\frac{\mathcal{P}(z)}{\Pi}\right)=1\right\}\right|\\
&\ll \max_{X_1\leq Y\leq 4X\atop 1\leq H\leq \log^{10}X}\frac{H}{Y}\cdot \sum_{1\leq h\leq \frac{X}{T}}\frac{h}{\varphi(h)}\left(\frac{Y}{H\log^2 z}+z^{\frac{1}{2}}\right)\\
&\ll \frac{1}{\log^2 z}\cdot\frac{X}{T},
\end{align*}
by the elementary bound $\sum_{m\leq M}\frac{m}{\varphi(m)}\ll M$. This proves the statement.\end{proof}

\section{Mean squares of Dirichlet polynomials}

With all the necessary lemmas available, we are ready to present the propositions that quickly lead to Theorem \ref{t4} and are also necessary in proving Theorem \ref{t5}.

\begin{proposition} \label{p1} Let $X\geq 1, T\geq T_0=X^{0.01},0\leq \alpha_1\leq 1$ and $1\leq P\ll X^{o(1)}$, where $P$ is a function of $X$. Define 
\begin{align*}
K(s)=\sum_{n\sim \frac{X}{P}}\frac{a_n}{n^s}\quad\text{and}\quad P(s)=\sum_{p\sim P}\frac{b_p}{p^s},
\end{align*}
where $a_n$ and $b_p$ are arbitrary complex numbers. Denoting
\begin{align*}
\mathcal{T}_1=\{t\in [T_0,T]:|P(1+it)|\leq P^{-\alpha_1}\}
\end{align*}
we have
\begin{align*}
\int_{\mathcal{T}_1}|K(1+it)P(1+it)|^2dt\ll \frac{T}{X}\cdot P^{1-2\alpha_1}\left(S_1\left(\frac{X}{P},(a_n)\right)+S_2\left(\frac{X}{P},(a_n)\right)\right).
\end{align*}
\end{proposition}

\begin{proof} The improved mean value theorem (Lemma \ref{2}) and definition of $\mathcal{T}_1$ give
\begin{align*}
\int_{\mathcal{T}_1}|K(1+it)P(1+it)|^2dt&\ll P^{-2\alpha_1}\int_{\mathcal{T}_1}|K(1+it)|^2dt\\
&\ll P^{-2\alpha_1}\left(T\sum_{k\sim  \frac{X}{P}}|a_k|^2+T\sum_{1\leq h\leq \frac{X}{PT}}\sum_{k,k'\sim \frac{X}{P}\atop k-k'=h}|a_k||a_{k'}|\right)\\
&\ll P^{-2\alpha_1}\left(\frac{TP}{X}S_1\left(\frac{X}{P},(a_n)\right)+\frac{TP}{X}S_2\left(\frac{X}{P},(a_n)\right)\right)\\
&= \frac{T}{X}\cdot P^{1-2\alpha_1}\left( S_1\left(\frac{X}{P},(a_n)\right)+S_2\left(\frac{X}{P},(a_n)\right)\right),
\end{align*}
which was the claim.\end{proof}

\begin{proposition} \label{p2} Let $X\geq 1, T\geq T_0=X^{0.01}$ and $1\leq P\ll X^{o(1)}$. Also let $0\leq \alpha_1,\alpha_2\leq 1$ and let the Dirichlet polynomials $K(s)$ and $M(s)$ with $K=\frac{X}{M}\gg X^{\varepsilon}$ be
\begin{align*}
K(s)=\sum_{n\sim K}\frac{a_n}{n^s}\quad\text{and}\quad M(s)=\sum_{m\sim M} \frac{c_m}{m^s},
\end{align*}
where $|c_m|\leq d_r(m)$ for fixed $r$, and $|a_n|= 1_{S}(n)$ for some set $S$ whose elements have at most $r$ prime factors from $[P,2P]$ and have no prime factors in $[1,X^{0.01}]\setminus \bigcup_{i=1}^{r}[Z_i,Z_i^2]$ for some $Z_i\geq 1$. Write 
\begin{align*}
P(s)&=\sum_{p\sim P}\frac{b_p}{p^s}\quad\text{with}\quad |b_p|\leq 1\quad \text{and}\\
\mathcal{T}&=\{t\in [T_0,T]:\,\,|P(1+it)|\geq P^{-\alpha_1}\,\, \text{and}\,\, |M(1+it)|\leq M^{-\alpha_2}\}.
\end{align*}
We have
\begin{align*}
\int_{\mathcal{T}}|K(1+it)M(1+it)|^2dt\ll M^{-2\alpha_2}P^{(2+10\varepsilon)\alpha_1\ell}\cdot (\ell!)^{1+o(1)}\cdot \left(\frac{T}{X}\cdot \frac{1}{\log X}+\frac{1}{\log^2 X}\right),
\end{align*}
where $\ell=\lceil\frac{\log \frac{X}{K}}{\log P}\rceil$. 
\end{proposition}

\begin{remark}
For products of three primes, our variables are picked so that the bound given by this proposition saves $X^{\varepsilon}$ over the trivial bound. However, for products of $k\geq 4$ primes, our savings are much more modest, and the factor $\frac{T}{X}\cdot \frac{1}{\log X}+\frac{1}{\log^2 X}$ becomes necessary.
\end{remark}

\begin{proof}  This result is inspired by Lemma 13 in \cite{matomaki}. Using the fact that $|M(1+it)|^2\leq M^{-2\alpha_2}(P^{\alpha_1}|P(1+it)|)^{2\ell}$ for $t\in \mathcal{T}$ and splitting polynomials into shorter ones, we have
\begin{align}\label{eq34}
\int_{\mathcal{T}}|K(1+it)M(1+it)|^2dt&\ll  M^{-2\alpha_2}P^{2\alpha_1\ell}\int_{\mathcal{T}}|K(1+it)P(1+it)^{\ell}|^2dt\nonumber\\
&\ll M^{-2\alpha_2}P^{2\alpha_1\ell}\ell^2\int_{ \mathcal{T}}|A(1+it)|^2dt,
\end{align}
where
\begin{align*}
A(s)=\sum_{n\sim Y}\frac{A_n}{n^s}
\end{align*}
for some $KP^{\ell}\leq Y\leq 2K(2P)^{\ell}$ (so $X\leq Y\leq 2^{\ell}PX$), the coefficients $A_n$ satisfying
\begin{align*}
|A_n|\leq\sum_{\substack{n=p_1\dotsm p_{\ell}m\\p_i\sim P\\m\sim K}}|a_m|.
\end{align*}
By the improved mean value theorem (Lemma \ref{2}), we see that $\eqref{eq34}$ is bounded by
\begin{align*}
\ll M^{-2\alpha_2}P^{2\alpha_1\ell}\ell^2\left(T\sum_{n\sim Y}\left|\frac{A_n}{n}\right|^2+T\sum_{1\leq h\leq \frac{Y}{T}}\sum_{m-n=h}\frac{|A_m||A_n|}{mn}\right)
\end{align*}
Note that $A_n\neq 0$ implies that $n$ has at most $\ell+r$ prime factors from $[P,2P]$ and that $n$ is coprime to
\begin{align*}
\Pi=\prod_{\substack{p\leq X^{0.01}\\p\not \in \bigcup_{i=1}^r[Z_i,Z_i^2]\cup [P,2P]}}p.
\end{align*}
Consequently, $|A_n|\leq (\ell+r)!$, and so
\begin{align*}
\sum_{n\sim Y}\left|\frac{A_n}{n}\right|^2&\leq \frac{1}{Y}\cdot (\ell+r)!\sum_{n\sim Y}\frac{|A_n|}{n}\\&\ll \frac{1}{Y}(\ell!)^{1+o(1)}\sum_{m\sim K}\frac{|a_m|}{m}\sum_{p_1,...,p_\ell\sim P}\frac{1}{p_1\dotsm p_{\ell}}\\
&\ll (\ell!)^{1+o(1)}\cdot \frac{1}{Y}\sum_{m\sim K\atop (m,\Pi)=1}\frac{|a_m|}{m}\\
&\ll (\ell!)^{1+o(1)}\cdot \frac{1}{X\log X},
\end{align*}
where the last step comes from Brun's sieve and the facts that $Y\geq X$ and $K\gg X^{\varepsilon}$.\\

To deal with the second sum arising from the improved mean value theorem, notice that by Brun's sieve the number of $n\leq y$ with $(n(kn+h),\Pi)=1$ is $\ll \frac{y}{\log^2 y}\frac{hk}{\varphi(hk)}$ with an absolute implied constant. Since $\varphi(ab)\geq \varphi(a)\varphi(b)$ and $\frac{k}{\varphi(k)}\leq 2^{\ell}$ when $k$ has $\ell$ prime factors, we have
\begin{align*}
&\sum_{1\leq h\leq \frac{Y}{T}}\sum_{n\sim Y}\frac{|A_n||A_{n+h}|}{n(n+h)}\\
&\leq  \frac{1}{Y^2}\cdot (\ell+r)!\sum_{1\leq h\leq \frac{Y}{T}}\sum_{p_1,...,p_{\ell}\sim P}\sum_{\substack{(m,\Pi)=1\\(p_1\dotsm p_{\ell}m+h,\Pi)=1\\m\leq \frac{2Y}{p_1\dotsm p_{\ell}}}}1\\
&\ll \frac{1}{Y^2}\cdot (\ell!)^{1+o(1)}\sum_{1\leq h\leq \frac{Y}{T}}\sum_{p_1,...,p_{\ell}\sim P}\frac{Y}{p_1\dotsm p_{\ell}\log^2 \frac{Y}{p_1\dotsm p_{\ell}}}\frac{p_1\dotsm p_{\ell}h}{\varphi(p_1\dotsm p_{\ell}h)}\\
&\ll \frac{1}{Y\log^2 Y}(\ell!)^{1+o(1)}\sum_{1\leq h\leq \frac{Y}{T}}\frac{h}{\varphi(h)}\sum_{p_1,...,p_{\ell}\sim P}\frac{1}{p_1\dotsm p_{\ell}}\\
&\ll \frac{1}{T}(\ell!)^{1+o(1)}\frac{1}{\log^2 X},
\end{align*}
as desired.\end{proof}

\begin{proposition}\label{p3} Let  $X^{1+o(1)}\geq T\geq T_0=X^{0.01}$ and $0\leq \alpha_1\leq 1$. Furthermore, let
\begin{align*}
P(s)=\sum_{p\sim P}\frac{a_p}{p^s},\quad \text{and}\quad M(s)=\sum_{M\leq q\leq M'}\frac{1}{q^s},
\end{align*}
with  $|a_p|\leq 1$, $M'\in [M+\frac{M}{\log P},2M]$, $\log X\leq P\ll X^{o(1)}$ and $PM=X^{1+o(1)}$, and 
let 
\begin{align*}
\mathcal{U}=\{t\in [T_0,T]:|P(1+it)|\geq P^{-\alpha_1}\}.
\end{align*}
Then, for $\ell=\lfloor \varepsilon\frac{\log X}{\log P}\rfloor$, 
\begin{align*}
&\int_{\mathcal{U}}|P(1+it)M(1+it)|^2dt\\
&\ll (P^{2\alpha_1-1}\log^2 X)^{(1+o(1))\ell}X^{o(1)}+(\log X)^{-100}\left(1+\frac{|\mathcal{U}'|T^{\frac{1}{2}}}{X^{\frac{2}{3}-o(1)}}\right)
\end{align*}
for some well-spaced set $\mathcal{U}'\subset \mathcal{U}$.
\end{proposition}

\begin{proof} Heath-Brown's decomposition (Lemma \ref{11}) with $k=3$ allows us to write, for some $C>0$,
\begin{align*}
|M(1+it)|\ll (\log^{C} X)(|G_1(1+it)|+\dotsm +|G_L(1+it)|)
\end{align*}
with $L\leq \log^C X$. Here each $G_j(s)$ is either of the form
\begin{align*}
G_j(s)=M_1(s)M_2(s)M_3(s),\,\, M_1M_2M_3=X^{1+o(1)},\,\, M_1\geq M_2\geq M_3,\,\, M_3\geq \exp\left(\frac{\log X}{2\log \log X}\right)
\end{align*}
with $M_i(s)$ prime-factored polynomials, or of the form
\begin{align*}
G_j(s)=N_1(s)N_2(s),\,\, N_1N_2=X^{1+o(1)},\,\, N_1\geq N_2
\end{align*}
with $N_i(s)$ zeta sums (it is possible that $N_2(s)$ is the constant polynomial $1^{-s}$). It suffices to bound the contributions of the zeta sums and the prime-factored polynomials separately.\\

We look at the zeta sums first. We split the integration domain into dyadic intervals $[T_1,2T_1]$ with $T_0\leq T_1\leq T$. Keeping in mind that $N_1\geq X^{\frac{1}{2}-o(1)}$, $P^{\ell}=X^{\varepsilon+o(1)},$ and $|P(1+it)P^{\alpha_1}|^{2\ell}\geq 1$ for $t\in \mathcal{U}$, Cauchy-Schwarz and Watt's theorem (Lemma \ref{16}) yield
\begin{align*}
&\int_{\mathcal{U}\cap [T_1,2T_1]}|P(1+it)N_1(1+it)N_2(1+it)|^2dt\\
&\ll P^{2\alpha_1\ell}\int_{\mathcal{U}\cap[T_1,2T_1]}|N_1(1+it)N_2(1+it)P(1+it)^{\ell}|^2dt\\
&\ll P^{2\alpha_1\ell}\left(\int_{T_1}^{2T}|N_1(1+it)|^4|P(1+it)|^{4\ell}dt\right)^{\frac{1}{2}}\cdot\left(\int_{T_1}^{2T_1}|N_2(1+it)|^4 dt\right)^{\frac{1}{2}}\\
&\ll P^{2\alpha_1 \ell}X^{o(1)}\left(\left(\frac{T_1+T_1^{\frac{1}{2}}P^{4\ell}}{N_1^2P^{2\ell}}+\frac{1}{T_1^3}\right)(2\ell)!^2\right)^{\frac{1}{2}}\cdot \left(\frac{T_1+N_2^2}{N_2^2}\right)^{\frac{1}{2}}\\
&\ll  P^{(2\alpha_1-1) \ell}X^{o(1)}\cdot (\ell!)^{2+o(1)}+\frac{P^{2\alpha_1\ell}X^{o(1)}(\ell!)^{2+o(1)}}{T_0}\\
&\ll (P^{2\alpha_1-1}\log^2 X)^{(1+o(1))\ell}X^{o(1)}+X^{-\varepsilon}.
\end{align*}
Combining the contributions of the dyadic intervals simply multiplies this bound by $\log X$.\\

To bound the contribution of the prime-factored polynomials, we first observe that
\begin{align*}
\int_{\mathcal{U}}|P(1+it)M(1+it)|^2 dt\ll \sum_{t\in \mathcal{U}'}|P(1+it)M(1+it)|^2
\end{align*}
for some well-spaced $\mathcal{U}'\subset \mathcal{U}$. We make use of the Halász-Montgomery inequality (Lemma \ref{18}), and of the prime-factored property applied to the polynomial $M_3$ with length $M_3\in \left[\exp\left(\frac{\log X}{2\log \log X}\right), X^{\frac{1}{3}+o(1)}\right]$, finding that
\begin{align*}
 &\sum_{t\in\mathcal{U}'}|P(1+it)M_1(1+it)M_2(1+it)M_3(1+it)|^2\\
 &\ll (\log X)^{-100-D}\sum_{t\in \mathcal{U}'}|P(1+it)M_1(1+it)M_2(1+it)|^2\\
 &\ll (\log X)^{-100-2C}\left(1+\frac{T^{\frac{1}{2}}|\mathcal{U}'|}{X^{\frac{2}{3}-o(1)}}\right),
\end{align*}
where $D$ is so large that $D-2C-1$ exceeds the power of logarithm arising from the mean square of the coefficients of the divisor-bounded polynomial $P(s)M_1(s)M_2(s)$. Now the statement is proved.\end{proof}

\section{Proof of Theorem 4}

The following proposition yields Theorem \ref{t4} (and hence Theorems \ref{t1} and \ref{t2}) immediately, in view of the remarks of Subsection \ref{subsec:reduction}

\begin{proposition}\label{p4} Let $k\geq 3$ be a fixed integer, $\varepsilon>0$ be small enough and $T_0=X^{0.01}$, as before. Define 
\begin{align*}
F(s)=\sum_{\substack{p_1\dotsm p_k\sim X\\P_i\leq p_i\leq P_i^{1+\varepsilon}\\
i\leq k-1}}(p_1\dotsm p_k)^{-s},
\end{align*}
where $P_i$ are as in Theorem \ref{t4}. Then, for $T\geq T_0$, we have
\begin{align}
\int_{T_0}^{T}|F(1+it)|^2dt\ll \left(\frac{TP_1\log X}{X}+1\right)\cdot \frac{1}{(\log^2 X)(\log_k X)^{3}}.
\end{align}
\end{proposition}

\begin{proof} We make use of the ideas introduced in the paper \cite{matomaki} by Matom\"aki and Radziwi\l{}\l{}. Trivially, we may assume $T\leq X^{1+o(1)}$. Let $H=(\log_k X)^{3}$,
\begin{align*}
Q_{v,H}(s)=\sum_{e^{\frac{v}{H}}\leq p<e^{\frac{v+1}{H}}}p^{-s},
\end{align*}
and for each $j=1,...,k$,
\begin{align*}
F_{v,H,j}(s)=\sum_{\substack{p_1\dotsm p_{j-1}p_{j+1}\dotsm p_k\sim Xe^{-\frac{v}{H}}\\P_i\leq p_i\leq P_i^{1+\varepsilon},\,i\neq j,\,i\leq k-1}}(p_1\dotsm p_{j-1}p_{j+1}\dotsm p_k)^{-s}.
\end{align*}
Define $\alpha_1,...,\alpha_{k-1}$ by $\alpha_{j}=10j\varepsilon$ for $j\leq k-2$, and $\alpha_{k-1}=\frac{1}{12}-\varepsilon$, with $\varepsilon$ so small that $\alpha_{k-2}\leq \frac{\sqrt{\varepsilon}}{10}$. We split the domain of integration as $[T_0,T]=\mathcal{T}_1\cup \mathcal{T}_2\cup\dotsm \cup \mathcal{T}_{k-1} \cup \mathcal{T}$. We write $t\in \mathcal{T}_1$ if
\begin{align*}
|Q_{v,H}(1+it)|\leq e^{-\frac{\alpha_1 v}{H}}
\end{align*}
for all $v\in I_1=[H\log P_1,(1+\varepsilon)H\log P_1]$. We define recursively $t\in \mathcal{T}_j$ for $j=2,...,k-1$ if $t\not \in \bigcup_{j'\leq j-1} \mathcal{T}_{j'}$ but
\begin{align*}
|Q_{v,H}(1+it)|\leq e^{-\frac{\alpha_j v}{H}}
\end{align*} 
for all $v\in  I_j=[H\log P_j,(1+\varepsilon)H\log P_j]$. Finally, we write
\begin{align*}
\mathcal{T}=[T_0,T]\setminus \bigcup_{j=1}^{k-1}\mathcal{T}_j.
\end{align*}
Lemma \ref{6}, with the notation of Subsection \ref{subsec:sieve}, yields
\begin{align}\label{eq35}
\int_{\mathcal{S}}|F(1+it)|^2dt&\ll H^2(\log^2 P_j)\int_{\mathcal{S}}|Q_{v_j,H}(1+it)F_{v_j,H,j}(1+it)|^2dt\nonumber\\
&+\frac{T}{HX}(S_1(X,(c_n))+S_2(X,(c_n)))
\end{align}
for some $v_j\in I_j$, and any $\mathcal{S}\subset [T_0,T]$. The coefficients $c_n$ in the definitions of $S_1$ and $S_2$ are naturally the convolution of the absolute values of the coefficients of the polynomials $Q_{v_j,H}(s)$ and $F_{v_j,H,j}(s).$ By Lemma \ref{15} and the remark related to it, the last two terms above contribute 
\begin{align*}
&\ll \frac{T}{X}\cdot \frac{1}{H\log X}+\frac{1}{H\log^2 X}\\
&\ll \left(\frac{TP_1\log X}{X}+1\right)\cdot \frac{1}{H\log^2 X}.
\end{align*}

We choose $\mathcal{S}=\mathcal{T}_1,...,\mathcal{T}_{k-1},\mathcal{T}$ in \eqref{eq35}. Summarizing, it suffices to estimate for each $j=1,...,k-1$ the quantity
\begin{align*}
B_j:=H^2(\log^2 P_j)\int_{\mathcal{T}_j}|Q_{v_j,H}(1+it)F_{v_j,H,j}(1+it)|^2dt,
\end{align*}
where $v_j\in [H\log P_j,(1+\varepsilon)H\log P_j]$ is chosen so that the integral is maximal, and additionally the quantity
\begin{align*}
B:=H^2(\log^2 X)\int_{\mathcal{T}}|Q_{v_k,H}(1+it)F_{v_k,H,k}(1+it)|^2 dt,
\end{align*}
where $v_k\in [H\log \frac{X}{(P_1\dotsm P_{k-1})^{1+\varepsilon}},H\log \frac{2X}{P_1\dotsm P_{k-1}}]$ is also picked so that the integral is maximized.\\

The integral over $\mathcal{T}_1$ is bounded with the help of Proposition \ref{p1}. We take $K(s)=F_{v_1,H,1}(s)$ and $P(s)=Q_{v_1,H}(s)$. Now Lemma \ref{15} and Proposition \ref{p1} result in
\begin{align*}
B_1&\ll H^2(\log^2 P_1) P_1^{1+\varepsilon-2\alpha_1}\frac{T}{X}\left(\frac{1}{\log X}+\frac{X}{P_1T}\cdot \frac{1}{\log^2 X}\right)\\
&\ll\left(\frac{TP_1\log X}{X}+1\right)\cdot \frac{P_1^{10\varepsilon-2\alpha_1}}{\log^2 X},
\end{align*}
and this is an admissible bound, since $\alpha_1=10\varepsilon$ and $P_1\gg(\log_{k} X)^{\varepsilon^{-1}}$.\\

For the integral over $\mathcal{T}_j$ with $2\leq j\leq k-1$ we use Proposition \ref{p2}, with $K(s)=F_{v_j,H,j}(s), M(s)=Q_{v_j,H}(s)$ and $P(s)=Q_{v_{j-1},H}(s)$, and for $\ell=\lceil \frac{\log P_j}{\log P_{j-1}}\rceil$ deduce
\begin{align}\label{eq44}
B_{j}&\ll H^2(\log^2 P_{j})P_{j}^{-2\alpha_{j}}\cdot P_{j-1}^{(2+10\varepsilon) \alpha_{j-1} \ell}\nonumber\\ 
&\quad\cdot (\ell!)^{1+o(1)}\cdot \left(\frac{T}{X\log X}+\frac{1}{\log^2 X}\right)\nonumber\\
&\ll P_{j-1}^{10}P_{j}^{2(\alpha_{j-1}-\alpha_{j})+10\varepsilon+(1+\varepsilon)\frac{\log \log P_{j}}{\log P_{j-1}}}\left(\frac{TP_1\log X}{X}+1\right)\frac{1}{\log^2 X}.
\end{align}
For $2\leq j\leq k-2$, we have $\frac{\log \log P_j}{\log P_{j-1}}\leq 2\varepsilon$ and $\alpha_j-\alpha_{j-1}=10\varepsilon$, so the definitions of $P_{j-1}$ and $P_j$ result in
\begin{align*}
B_j\ll \left(\frac{TP_1\log X}{X}+1\right)\frac{1}{\log^2 X}(\log_k X)^{-3},
\end{align*}
as wanted. For $j=k-1$, we have $\alpha_{k-2}\leq \frac{\sqrt{\varepsilon}}{10}$, $\alpha_{k-1}=\frac{1}{12}-\varepsilon$ and $P_{k-1}=(\log X)^{\varepsilon^{-2}}$, so taking $j=k-1$ in the above computation gives
\begin{align*}
B_{k-1}&\ll P_{k-1}^{-\frac{1}{6}+\frac{1}{4}\sqrt{\varepsilon}+\frac{1+\varepsilon}{6+10\sqrt{\varepsilon}}}\ll P_{k-1}^{-\varepsilon}\ll (\log X)^{-\varepsilon^{-1}},
\end{align*}
and therefore the case of $\mathcal{T}_{k-1}$ has been dealt with.\\

Finally, the integral over $\mathcal{T}$ is estimated using Proposition \ref{p3} with $P(s)=Q_{v_{k-1},H}(s)$ and $M(s)=Q_{v_k,H}(s)$. Denoting $\ell=\lfloor \varepsilon\frac{\log X}{\log P_{k-1}}\rfloor$ and separating by Perron's formula the variable $p_{k-1}$ from the rest of the variables in $F_{v_k,H,k}(s)$ (and bounding the polynomial corresponding to the variables $p_1,...,p_{k-2}$ by $\ll 1$), we see that
\begin{align*}
B&\ll H^2(\log^4 X)\int_{\mathcal{T}}|Q_{v_{k-1},H}(1+it)Q_{v_k,H}(1+it)|^2dt\\
&\ll H^2(\log^4 X)(P_{k-1}^{-\frac{5}{6}+2\varepsilon}\log^2 X)^{(1+o(1))\ell}X^{o(1)}+(\log X)^{-95}\left(1+\frac{|\mathcal{T'}|T^{\frac{1}{2}}}{X^{\frac{2}{3}-o(1)}}\right)
\end{align*}
for some well-spaced set $\mathcal{T}'\subset \mathcal{T}$. Since $P_{k-1}=(\log X)^{\varepsilon^{-2}}$, the first term is $\ll X^{-\frac{\varepsilon}{3}}$. In addition, Lemma \ref{7} allows us to bound the size of $\mathcal{T}'$ by
\begin{align*}
|\mathcal{T}'|\ll T^{2\alpha_{k-1}}P_{k-1}^2 X^{(\varepsilon^2+o(1))}\ll X^{\frac{1}{6}-\frac{\varepsilon}{2}}, 
\end{align*}
because $\alpha_{k-1}=\frac{1}{12}-\varepsilon$. Therefore, the integral over $\mathcal{T}$ is $\ll (\log X)^{-95}$. In conclusion, we deduced the bound 
\begin{align*}
B_1+\dotsm +B_{k-1}+B\ll \left(\frac{TP_1\log X}{X}+1\right)\cdot \frac{1}{H\log^2 X},
\end{align*}
which finishes the proof of this proposition and of Theorem \ref{t4}.\end{proof}

\subsection{A corollary on products of two primes}

As a byproduct of the methods above, we arrive at the exponent $c=5+\varepsilon$ for products of two primes, which already replicates Mikawa's exponent for $P_2$ numbers\footnote{Adding to the argument a small refinement from Subsection \ref{subsec:Exp}, as well as Proposition \ref{p8}, which is rather similar to Proposition \ref{p3}, would already give $c$ somewhat smaller than $5$.} . Similarly as for products of three or more primes, it suffices to prove 
\begin{align*}
\int_{T_0}^{T}|F(1+it)|^2dt=o\left(\left(\frac{TP_1\log X}{X}+1\right)\cdot \frac{1}{(\log X)^{2+\varepsilon}}\right),
\end{align*}
where
\begin{align*}
F(s)=\sum_{\substack{p_1p_2\sim X\\P_1\leq p_1<P_1^{1+\varepsilon}}} (p_1p_2)^{-s},
\end{align*}
and $P_1=\log^a X$ with $a=4+\varepsilon.$ We may again suppose $T\leq X^{1+o(1)}$.\\ 

We can redefine the set $\mathcal{T}_1$ in the proof of Proposition \ref{p4} with the new values $P_1=\log^{a} X$, $H=(\log X)^{3\varepsilon}$, keeping the value $\alpha_1=10\varepsilon$, and we see again from Proposition \ref{p1} that the mean square of $F(1+it)$ over $\mathcal{T}_1$ is suitably small. For applying Propositions \ref{p2} and \ref{p3}, we need more polynomials than the two that correspond to the variables $p_1$ and $p_2$ in \eqref{eq25}. Indeed, Heath-Brown's decomposition (Lemma \ref{11}) enables splitting the polynomial corresponding to $p_2$ as $(\log X)^{O(1)}$ sums of the form $|M_1(s)M_2(s)|+|N_1(s)N_2(s)|$, where $M_1(s)$ and $M_2(s)$ are prime-factored Dirichlet polynomials with $M_1M_2=X^{1+o(1)}$, $\exp\left(\frac{\log X}{2\log \log X}\right)\ll M_1 \ll X^{\frac{1}{3}+o(1)}$ and $N_1(s)$ and $N_2(s)$ zeta sums with $N_1N_2=X^{1+o(1)}$. The contribution of the zeta sums over the complement of $\mathcal{T}_1$ can be managed easily with Watt's theorem, similarly as in the proof of Proposition \ref{p3}.\\

To estimate the contribution of the prime-factored polynomials $M_i(s)$, we redefine the set $\mathcal{T}_2$ as  $\{t\in [T_0,T]: |M_1(1+it)|\leq M_1^{-\alpha_2}\}\setminus \mathcal{T}_1$, and Proposition \ref{p2} (with $P(s)$ corresponding to $p_1$ and $K(s)=M_1(s)M_2(s)$) produces a valid bound\footnote{This bound for $a$ arises by inserting $P_{j-1}=\log^a X$ and $P_j=X^{1+o(1)}$ into formula \eqref{eq44}.}\,\, in the $\mathcal{T}_2$ case, as long as $a\geq \frac{1}{2(\alpha_2-\alpha_1)}+100\varepsilon$. We take $\alpha_2=\frac{1}{8}-\varepsilon$, which turns out to be the best choice here.\\

Finally, when considering the integral over the complement of $\mathcal{T}_1\cup \mathcal{T}_2$, instead of Proposition \ref{p3} we apply the simple inequality
\begin{align*}
\int_{\mathcal{T}}|M_1(1+it)M_2(1+it)|^2dt\ll (\log X)^{-100}\left(1+\frac{|\mathcal{T}'|T^{\frac{1}{2}}}{M_2}\right)
\end{align*}
for some well-spaced $\mathcal{T}'\subset \mathcal{T}$, with $\mathcal{T}\subset [T_0,T]$ arbitrary. This inequality follows just from the prime-factored property of $M_1(s)$ combined with the Halász-Montgomery inequality (Lemma \ref{18}). Now, denoting $M_1=X^{\nu+o(1)}$, we need to have $|\mathcal{T}'|\ll X^{\frac{1}{2}-\nu-\varepsilon^2}$ whenever
\begin{align*}
\mathcal{T}'\subset \{t\in [T_0,T]: |M_1(1+it)|\geq M_1^{-\alpha_2}\}
\end{align*}
is well spaced. Jutila's large values theorem (Lemma \ref{13}) applied with $F(s)=M_1(s)^{\ell}$, $V=M_1^{-(\frac{1}{8}-\varepsilon)\ell}$ and $k=2$, $\ell\in \{2,3\}$ tells that
\begin{align*}
|\mathcal{T'}|\ll \begin{cases}X^{\max\{\frac{\nu}{2},\,\,-\frac{11}{4}\nu+1,\,\,1-4\nu\}-2\varepsilon^{2}} \\
X^{\max\{\frac{3}{4}\nu,\,\,-\frac{33}{8}\nu+1,\,\,1-6\nu\}-2\varepsilon^{2}}.\end{cases}
\end{align*}
We know that $\nu\leq \frac{1}{3}+o(1)$, and for $\frac{2}{7}\leq \nu\leq \frac{1}{3}$ the first bound is $\ll X^{\frac{1}{2}-\nu-\varepsilon^2}$, while for $\frac{4}{25}\leq \nu\leq \frac{2}{7}$ the second bound is small enough.\\

In the case $\nu\leq \frac{4}{25}$, we may simply appeal to Lemma \ref{7} to bound $|\mathcal{T}'|$ (with $V=M_1^{-\alpha_2}$), and get
\begin{align*}
|\mathcal{T}'|\ll T^{2\alpha_2}X^{2\nu \alpha_2+o(1)}\ll X^{0.29+100\varepsilon}\ll X^{\frac{1}{2}-\nu-\varepsilon}
\end{align*}
for $\alpha_2=\frac{1}{8}-\varepsilon$. This proves that $\alpha_2=\frac{1}{8}-\varepsilon$ was permissible, leading to $a=\frac{1}{2\alpha_2}+C_1\varepsilon$, so the admissible exponent becomes $c=a+1\leq 5+2C_1\varepsilon$ (and $\varepsilon>0$ was arbitrary). The rest of the paper therefore deals with improving the value $c=5+\varepsilon$ to $c=3.51$, which will require several further ideas, along with the ones already introduced.

\section{Lemmas for Theorem 5}

\subsection{Exponent pairs}\label{subsec:Exp}

In the proof of Theorem \ref{t5}, several zeta sums arise, and in some instances it is useful to have a smallish, pointwise power saving in these sums. This is given by the theory of exponent pairs. We could compute a long list of exponent pairs and choose the optimal estimate depending on the length of the zeta sum, but it turns out that using a single suitable exponent pair improves the exponent $c$ for $E_2$ numbers by approximately $0.02$, while having more of them would have very little additional advantage, and would complicate the calculations. Therefore, instead of formulating the general definition of exponent pairs (found in \cite[Chapter 3]{montgomery}), we write down the estimate coming from this specific pair.

\begin{lemma} Let
\begin{align*}
\sigma(\nu)=-\min\left\{\frac{1-\nu}{126}-\frac{\nu}{21},0\right\}.
\end{align*}
Then we have
\begin{align*}
\sum_{n\in I}n^{-1-it}\ll t^{-\sigma(\nu)+o(1)}
\end{align*}
for each $I=[N_1,N_2]$ with $t^{\nu}\leq N_1\leq N_2\ll t^{\nu+o(1)}$. 
\end{lemma}

\begin{proof} This follows immediately from the fact that $(\frac{1}{126},\frac{20}{21})$ is an exponent pair. For the proof of this, see Montgomery's book \cite[Chapter 3]{montgomery}.\end{proof}

\subsection{Lemmas on sieve weights}

For finding products of two primes on short intervals, we need some lemmas concerning sieve weights. In the cases of sums $\Sigma_1(h)$ and $\Sigma_2(h)$ in Subsection \ref{subsec:Sigma1}, there will be too few variables for finding cancellation in the mean square of the corresponding Dirichlet polynomials. However, introducing sieve weights to these sums, we get an additional variable which is summed over all integers in a certain range, and separating that variable gives a long zeta sum (because there are few variables), and Watt's theorem can be applied to this sum. Also in the case of these sums, we need to make use of an additional saving of a logarithm in the mean value theorem. However, here the coefficients are not supported on almost primes but are closely related to the Dirichlet convolution $\lambda_n* 1$, where $\lambda_n$ are the sieve weights. The sieve weights $\lambda_n$ can be taken to be those of Brun's pure sieve. Specifically, we take
\begin{align*}
\lambda_d^{+}=\begin{cases}\mu(d),\quad \nu(d)\leq R,d\mid \mathcal{P}(w)\\ 0\quad\quad\quad \text{otherwise}\end{cases}\quad\quad \lambda_d^{-}=\begin{cases}\mu(d),\quad \nu(d)\leq R+1,d\mid \mathcal{P}(w)\\ 0\quad\quad\quad \text{otherwise}\end{cases}
\end{align*} 
where the notations are as in Subsection \ref{subsec:notation}, and 
\begin{align*}
 w=\exp\left(\frac{\log X}{(\log \log X)^3}\right)\quad \text{and}\quad R=2\left\lfloor (\log \log X)^{\frac{3}{2}}\right\rfloor.
\end{align*}
Since the support of $\lambda_n*1$ contains in addition to almost primes only numbers having exceptionally many prime factors, we are able to save one logarithm factor in the mean values. This is done in the following lemma.

\begin{lemma} \label{5}Let $\lambda_d^{+}$ and $\lambda_d^{-}$ be the sieve weights of Brun's pure sieve with the above notations. Let $k\geq 0$ be a fixed integer, $R_1,...,R_k\geq 1$ and 
\begin{align*}
a_n=\sum_{p_1\dotsm p_k\mid n\atop R_i\leq p_i\leq R_i^{1+\varepsilon}}\left|\sum_{n=p_1\dotsm p_k dm}\lambda_d^{\pm}\right|
\end{align*}
where either the sign $+$ or $-$ is chosen throughout (for $k=0$, we define $p_1\dotsm p_k=1$). Then for $y\gg_A \frac{x}{\log ^{A}x}$ and $x\sim X$ we have
\begin{align}
&\sum_{x\leq n\leq x+y}|a_n|^2\ll_{A} (\log \log X)^{O_k(1)}\frac{y}{\log X}\label{eq6}\\
&\sum_{1\leq h\leq \frac{x}{T}}\sum_{m-n=h\atop m,n\in [x,x+y]}|a_m||a_n|\ll_{A}(\log \log X)^{O_k(1)} \frac{X}{T}\cdot \frac{y^2}{\log^2 X}\label{eq45}.
\end{align}\end{lemma}

For the proof of this lemma, we need a couple of other lemmas.

\begin{lemma}\label{4} For $x\geq 2$ and positive integer $\ell$, let 
\begin{align*}
\pi_{\ell}(x)=|\{n\in [1,x]:\nu(n)=\ell\}|.
\end{align*}
There exist absolute constants $K$ and $C$ such that
\begin{align*}
\pi_{\ell}(x)<\frac{Kx}{\log x}\frac{\left(\log \log x+C\right)^{\ell-1}}{(\ell-1)!}
\end{align*}
for all $\ell$ and $x\geq 2$.\end{lemma}

\begin{proof} This is an elementary result of Hardy and Ramanujan from \cite{hardy}.\end{proof}

\begin{lemma}\label{8} Let $a\geq 1$ be fixed, and let $R=2\lfloor (\log \log X)^{\frac{3}{2}}\rfloor$ as before. Then for any $A>0$
\begin{align*}
\sum_{n\sim X\atop \nu(n)\geq R}a^{\nu(n)}\ll_{a,A} \frac{X}{\log^{A} X}.
\end{align*}\end{lemma}

\begin{proof} The sum in question can be written as
\begin{align*}
\sum_{\ell\geq R}a^{\ell}|\{n\sim X:\nu(n)=\ell\}|,
\end{align*}
and by Lemma \ref{4} this is
\begin{align*}
&\ll \frac {X}{\log X}\sum_{\ell\geq R}\left(\frac{ae(\log \log X+C)}{\ell-1}\right)^{\ell-1}\\
&\ll_{a} X\cdot 2^{-R}\ll_A \frac{X}{\log^A X}
\end{align*}
by the definition of $R$.\end{proof}

We can now proceed to proving Lemma \ref{5}.\\

\begin{proof}[Proof of Lemma \ref{5}.] It suffices to consider the lower bound sieve weights. We assume $k\geq 1$, as the case $k=0$ is similar but a little simpler. Define $\theta_n=1*\lambda_n^{-}$. We have 
\begin{align*}
\theta_n&=\sum_{\substack{d\mid n\\\nu(d)\leq R\\d\mid \mathcal{P}(w)}}\mu(d)\nonumber\\
&=\sum_{d\mid (n,\mathcal{P}(w))}\mu(d)+O\left(\sum_{d\mid n\atop \nu(d)>R}|\mu(d)|\right)\nonumber\\
&=1_{(n,\mathcal{P}(w))=1}+O(2^{\nu(n)}1_{\nu(n)>R}).
\end{align*}
Using this, we bound the sum \eqref{eq6}. Denoting by $\Pi$ the product of all the primes in $\bigcup_{i=1}^{k}[R_i,R_i^{1+\varepsilon}]\cap[1,w]$, we observe that
\begin{align}\label{eq37}
a_n&=\sum_{p_1\dotsm p_k\mid n\atop R_i\leq p_i\leq R_i^{1+\varepsilon}}|\theta_{\frac{n}{p_1\dotsm p_k}}|\leq \nu(n)^k(1_{\left(n,\frac{\mathcal{P}(w)}{\Pi}\right)=1}+2^{\nu(n)}1_{\nu(n)>R}).
\end{align}
The contribution of the first term on the right-hand side of \eqref{eq37} to the sum \eqref{eq6} is
\begin{align*}
&\ll \sum_{x\leq n\leq x+y\atop \left(n,\frac{\mathcal{P}(w)}{\Pi}\right)=1} \nu(n)^{2k}\ll (\log \log X)^{O_k(1)}\sum_{x\leq n\leq x+y\atop \left(n,\frac{\mathcal{P}(w)}{\Pi}\right)=1}1\ll (\log \log X)^{O_k(1)}\frac{y}{\log X}
\end{align*}
by Brun's sieve and the fact that $\nu(n)\ll (\log \log X)^3$ when $(n,\mathcal{P}(w))=1.$ On the other hand, the the second term on the right-hand side of \eqref{eq37} contributes to \eqref{eq6} at most
\begin{align}
&\ll \sum_{x\leq n\leq x+y\atop \nu(n)\geq R}\nu(n)^{2k}4^{\nu(n)}\ll_k \sum_{x\leq n\leq x+y\atop \nu(n)\geq R}5^{\nu(n)}\ll_{A,k}\frac{X}{\log^A X}\label{eq46}
\end{align}
by Lemma \ref{8}. This proves the first bound in Lemma \ref{5}.\\

The second bound in Lemma \ref{5} is proved analogously. The two terms in \eqref{eq37} can be combined in four ways into products of two terms (two of these are symmetric). One of the cases contributes to \eqref{eq45} at most
\begin{align*}
\ll \sum_{1\leq h\leq \frac{x}{T}}\sum_{m-n=h\atop m,n\in [x,x+y]}\nu(m)^k\nu(n)^k1_{\left(m,\frac{\mathcal{P}(w)}{\Pi}\right)=1}1_{\left(n,\frac{\mathcal{P}(w)}{\Pi}\right)=1}\ll (\log \log X)^{O_k(1)} \frac{X}{T}\cdot \frac{y^2}{\log^2 X}
\end{align*}
by Brun's sieve. The two symmetric terms obtained by multiplying terms in \eqref{eq37} have an impact of
\begin{align*}
\ll \sum_{1\leq h\leq \frac{x}{T}}\sum_{m-n=h\atop m,n\in [x,x+y]}\nu(m)^k\nu(n)^k1_{\left(m,\frac{\mathcal{P}(w)}{\Pi}\right)=1}2^{\nu(n)}1_{\nu(n)>R},
\end{align*}
where the coefficients depending on $m$ can be bounded trivially, while the coefficients depending on $n$ save an arbitrary power of logarithm, as in formula \eqref{eq46}. Finally, the fourth term arising from multiplication of  \eqref{eq37} also saves an arbitrary power of logarithm by the same argument.\end{proof}

\section{Proof of Theorem 5}

Before proving Theorem \ref{t5}, we need some preparation. Define
\begin{align*}
S_h(x)=\sum_{x\leq p_1p\leq x+h\atop P_1\leq p_1\leq P_1^{1+\varepsilon}}1,\quad S_X=S_X(X),
\end{align*}
and set
\begin{align*}
w=\exp\left(\frac{\log X}{(\log \log X)^3}\right).
\end{align*}
We use Buchstab's identity twice to decompose 
\begin{align*}
S_h(x)&=\sum_{\substack{x\leq p_1n\leq x+h\\P_1\leq p_1\leq P_1^{1+\varepsilon}\\(n,\mathcal{P}(w))=1\\n>1}}1-\sum_{\substack{x\leq p_1q_1n\leq x+h\\ P_1\leq p_1\leq P_1^{1+\varepsilon}\\w\leq q_1<\sqrt {x}\\(n,\mathcal{P}(q_1))=1\\n>1}}1\\
&=\sum_{\substack{x\leq p_1n\leq x+h\\P_1\leq p_1\leq P_1^{1+\varepsilon}\\(n,\mathcal{P}(w))=1\\n>1}}1-\sum_{\substack{x\leq p_1q_1n\leq x+h\\ P_1\leq p_1\leq P_1^{1+\varepsilon}\\w\leq q_1<\sqrt {x}\\(n,\mathcal{P}(w))=1\\n>1}}1+\sum_{\substack{x\leq p_1q_1q_2n\leq x+h\\ P_1\leq p_1\leq P_1^{1+\varepsilon}\\w\leq q_2<q_1<\sqrt {x}\\(n,\mathcal{P}(q_2))=1\\n>1}}1.
\end{align*}
Call these sums $\Sigma_1(h),\Sigma_2(h)$ and $\Sigma_3(h)$, respectively, and call the corresponding dyadic sums $\Sigma_1(X),\Sigma_2(X)$ and $\Sigma_3(X)$, respectively. We will divide $\Sigma_3(h)$ into two parts $\Sigma_3'(h)$ and $\Sigma_3''(h)$ in such a way that $\Sigma_1(h), \Sigma_2(h)$ and $\Sigma_3'(h)$ can be evaluated asymptotically, while the error from $\Sigma_3''(h)$ is manageable. To be precise, we will prove that
\begin{align}
\frac{1}{h}S_h(x)&=\frac{1}{h}(\Sigma_1(h)-\Sigma_2(h)+\Sigma_3'(h)+\Sigma_3''(h))\nonumber\\
&=\frac{1}{X}(\Sigma_1(X)-\Sigma_2(X)+\Sigma_3'(X))+\frac{1}{h}\Sigma_3''(h)+o\left(\frac{1}{\log X}\right)\label{eq38}\\
&=\frac{1}{X}S_X+\frac{1}{h}\Sigma_3''(h)-\frac{1}{X}\Sigma_3''(X)+o\left( \frac{1}{\log X}\right)\nonumber\\
&\geq \frac{1}{X}S_X-\frac{1}{X}\Sigma_3''(X)+o\left( \frac{1}{\log X}\right)\nonumber\\
&\geq \varepsilon\cdot \frac{1}{X}S_X\label{eq39}
\end{align}
almost always, with the steps \eqref{eq38} and \eqref{eq39} being the nontrivial ones. This estimate will then immediately  lead to Theorem \ref{t5}. To prove these statements, we require some auxiliary results for the cases of $\Sigma_1(h),\Sigma_2(h)$ and $\Sigma_3(h)$.

\subsection{Mean square bounds related to Theorem 5}

We need three additional mean square bounds to deal with the sums $\Sigma_1(h),\Sigma_2(h)$ and $\Sigma_3(h)$. The first is a relative of Proposition \ref{p3} and would already improve slightly the exponent $c=5+\varepsilon$ obtained from the proof of Theorem \ref{t4}. It will not be applied directly in the proof of Theorem \ref{t5}, but instead as an ingredient in the proof of Proposition \ref{p6}. 

\begin{proposition}\label{p8} Let  $X^{1+o(1)}\geq T\geq T_0=X^{0.01}$, and $0\leq \alpha_1\leq 1$. Furthermore, let
\begin{align*}
P(s)=\sum_{P\leq p\leq P'}\frac{1}{p^s},\quad M(s)=\sum_{m\sim M}\frac{b_m}{m^s},
\end{align*}
with $P=X^{\nu+o(1)}$, $P'\in \left[P+\frac{P}{\log X},2P\right]$, $0< \nu\leq \frac{1}{2}$, $|b_m|\leq d_r(m)$ for fixed $r$, and $PM=X^{1+o(1)}$. Also let 
\begin{align*}
\mathcal{U}=\{t\in [T_0,T]:|P(1+it)|\geq P^{-\alpha_1}\}.
\end{align*}
Then, 
\begin{align*}
\int_{\mathcal{U}}|P(1+it)M(1+it)|^2dt\ll (\log X)^{-100}+X^{\frac{1}{2}-\min\{2\sigma(\nu),\frac{\nu}{2}\}+o(1)}\cdot \frac{|\mathcal{U}'|P}{X}
\end{align*}
for some well-spaced $\mathcal{U}'\subset \mathcal{U}$.
\end{proposition}

\begin{proof}  Note that Heath-Brown's decomposition (Lemma \ref{11}) gives
\begin{align*}
|P(1+it)| \ll (\log^C X)(|G_1(1+it)|+\dotsm+|G_L(1+it)|)
\end{align*}
with $L\leq \log^C X$ and each $G_j(s)$ either of the form $G_j(s)=N(s)$ with $N(s)$ a zeta sum of length $P^{1-o(1)}$, or $G_j(s)=M_1(s)M_2(s)$ with $M_1$ and $M_2$ prime-factored polynomials of length $M_1\geq M_2\geq \exp(\frac{\log X}{\log \log X}),M_1M_2=P^{1-o(1)}.$ To bound the contribution of the zeta sum, we divide the integral over $\mathcal{U}$ into integrals over dyadic intervals $[T_1,2T_1]$ with $T_1\in [T_0,T]$, and write $N=T_1^{\mu+o(1)}$ with $\mu\geq \nu$. If $\mu>1$, we know that $|N(1+it)|\ll \frac{\log t}{t}$ and $M(1+it)\ll (\log X)^{O(1)}$, so
\begin{align*}
\int_{\mathcal{U}\cap [T_1,2T_1]}|M(1+it)N(1+it)|^2dt\ll \frac{(\log X)^{O(1)}}{T_0}.
\end{align*}
If $\mu\leq 1$, we first pick a well-spaced $\mathcal{U}'\subset \mathcal{U}$ such that
\begin{align*}
\int_{\mathcal{U}}|M(1+it)N(1+it)|^2dt\ll \sum_{t\in \mathcal{U}'}|M(1+it)N(1+it)|^2.
\end{align*}
Now the Halász-Montgomery inequality and the the fact that $N(s)$ is a zeta sum give
\begin{align*}
\sum_{t\in\mathcal{U}'\cap [T_1,2T_1]}|M(1+it)N(1+it)|^2&\ll T^{-2\sigma(\nu)+o(1)}\sum_{t\in\mathcal{U}'\cap [T_1,2T_1]}|M(1+it)|^2\\
&\ll T^{-2\sigma(\nu)+o(1)}\left(1+\frac{|\mathcal{U}'|T_1^{\frac{1}{2}+o(1)}}{\frac{X}{P}}\right).
\end{align*}

To deal with the contribution of the prime-factored polynomials $M_i(s)$, we may use the Halász-Montgomery inequality in a manner analogous to the above to obtain the estimate
\begin{align*}
&\int_{\mathcal{U}}|M(1+it)M_1(1+it)M_2(1+it)|^2dt\ll (\log X)^{-100}\left(1+\frac{|\mathcal{U}'|T^{\frac{1}{2}+o(1)}}{\frac{X}{P^{\frac{1}{2}}}}\right),
\end{align*}
since $MM_1\gg \frac{X^{1+o(1)}}{P^{\frac{1}{2}}}$
Taking the maximum of these two results produces the claimed bound. \end{proof} 

Our second mean square bound is a type I estimate where we exploit a long zeta sum with the help of Watt's theorem. In the cases of $\Sigma_1(h)$ and $\Sigma_2(h)$, this is necessary, and in the case $\Sigma_3(h)$ it improves our exponent for Theorem \ref{t5}. A closely related estimate can be found for example in \cite[Chapter 9]{harman-sieves}.

\begin{proposition} \label{p5} Let  $X^{1+o(1)}\gg T\geq T_0$, and let $M(s), N(s),P(s)$ be Dirichlet polynomials with coefficients bounded by $X^{o(1)}$ and supported on the intervals $[M,2M],[N,2N]$,$[P,2P]$, respectively. Denote $Q(s)=\sum_{m\sim Q}\frac{a_m}{m^s}$, and let $N(s)$ be a zeta sum. Suppose in addition that
\begin{align*}
MNP=X^{1+o(1)},\,\, PQ^2\leq X^{\frac{1}{4}},\,\,M^2P\ll X^{1+o(1)}.
\end{align*}
Then 
\begin{align*}
\int_{T_0}^{T}|M(1+it)N(1+it)P(1+it)Q(1+it)|^2dt\ll X^{o(1)} \left(Q^{-1}+\frac{1}{T_0}\right)\max_{m\sim Q}|a_m|^2.
\end{align*}
\end{proposition}

\begin{remark}
In all our applications, the polynomial $Q(s)$ has length essentially $X^{\varepsilon}$, and it is used to win by $X^{\varepsilon^2}$, say, in our estimates. 
\end{remark}

\begin{proof}  We will reduce the proposition to Watt's theorem (Lemma \ref{16}). Divide the integration domain into dyadic intervals $[T_1,2T_1]$. By Cauchy-Schwarz, the mean value theorem and Watt's theorem, we see that
\begin{align*}
&\int_{T_1}^{2T1}|M(1+it)N(1+it)P(1+it)Q(1+it)|^2dt\\
&\ll \left(\int_{T_1}^{2T1}|N(1+it)|^4|P(1+it)Q(1+it)^2|^2 dt\right)^{\frac{1}{2}}\\
&\quad\cdot\left(\int_{T_1}^{2T_1}|M(1+it)|^4 |P(1+it)|^2 dt\right)^{\frac{1}{2}}\\
&\ll \left(\left(\frac{T_1^{o(1)}(T_1+T_1^{\frac{1}{2}}P^2Q^4)}{N^2PQ^2}+\frac{T^{o(1)}}{T_1^3}\right)\max_{m\sim Q}|a_m|^4\right)^{\frac{1}{2}}\left(\frac{T_1+M^2P}{M^2P}\right)^{\frac{1}{2}}\\
&\ll\left(\left(\frac{T_1^{o(1)}(T_1+T_1^{\frac{1}{2}}P^2Q^4)}{N^2PQ^2}\right)\max_{m\sim Q}|a_m|^4\right)^{\frac{1}{2}}\left(\frac{T_1+M^2P}{M^2P}\right)^{\frac{1}{2}}+\frac{X^{o(1)}}{T_0}\max_{m\sim Q}|a_m|^2.
\end{align*}
Hence, we need 
\begin{align*}
(X+X^{\frac{1}{2}}P^2Q^4)(X+M^2P)\ll(MNPX^{o(1)})^2,
\end{align*}
and this is guaranteed by our conditions.\end{proof}

For the $\Sigma_3(h)$ case in Subsection \ref{subsec:sigma}, we also need the following mean square bound, which is somewhat analogous to Proposition \ref{p4} and is based on Propositions \ref{p1}, \ref{p2} and \ref{p8}, but it will be clear only later how it is crucial for proving Theorem \ref{t5}.

\begin{proposition} \label{p6}Let $0\leq \nu\leq \frac{1}{2},$ $0<\alpha_2\leq 1$, $a=\frac{1}{2\alpha_2}+C_2\varepsilon$, $P_1=\log^a X$, $X^{1+o(1)}\gg T\geq T_0=X^{0.01}$, and $w\leq P_2= X^{\nu+o(1)}$ with $w=\exp\left(\frac{\log X}{(\log \log X)^3}\right)$. Also let
\begin{align*}
G(s)=\sum_{\substack{p_1p_2p_3n\sim X\\P_i\leq p_i\leq P_i^{1+\varepsilon},\,i\leq 2\\p_2<p_3\\(n,\mathcal{P}(p_2))=1\\n>1}}a_n(p_1p_2p_3n)^{-s},
\end{align*}
where $|a_n|\ll (\log X)^{\varepsilon}$. Suppose that for every Dirichlet polynomial $M(s)=\sum_{m\sim M}\frac{b_m}{m^s}$ with $|b_m|\leq d_r(m)$ for fixed $r$ and $M=X^{\nu+o(1)}$ any well-spaced set
\begin{align*}
\mathcal{U}'\subset\{t\in [0,T]:|M(1+it)|\geq  M^{-\alpha_2}\}
\end{align*}
satisfies $|\mathcal{U}'|\ll X^{\frac{1}{2}-\nu+\min\{2\sigma(\nu),\frac{\nu}{2}\}-\varepsilon}$.
Then we have
\begin{align*}
\int_{T_0}^{T}|G(1+it)|^2 dt\ll \left(\frac{TP_1\log X}{X}+1\right)\frac{1}{\log^{2+\varepsilon} X}.
\end{align*}
\end{proposition}

\begin{proof} Let $\alpha_1=100\varepsilon$ and define $H=\log^{10\varepsilon}X$. Let
\begin{align*}
Q_{v,H,1}(s)=\sum_{e^{\frac{v}{H}}\leq p_1<e^{\frac{v+1}{H}}}p_1^{-s},\quad Q_{v,H,2}(s)=\sum_{e^{\frac{v}{H}}\leq p_2<e^{\frac{v+1}{H}}}p_2^{-s}
\end{align*}
and 
\begin{align*}
G_{v,H,1}(s)&=\sum_{\substack{p_2p_3p_4m\sim Xe^{-\frac{v}{H}}\\P_{2}\leq p_{2}\leq P_{2}^{1+\varepsilon}\\p_2<p_3,\,p_2\leq p_4\\(m,\mathcal{P}(p_4))=1}}a_{p_4m}(p_2p_3p_4m)^{-s},\\ G_{v,H,2}(s)&=\sum_{\substack{p_1p_3p_4m\sim Xe^{-\frac{v}{H}}\\P_{1}\leq p_1\leq P_1^{1+\varepsilon}\\(m,\mathcal{P}(p_4))=1}}a_{p_4m}(p_1p_3p_4m)^{-s}.
\end{align*}

For $j=1,2$, we have
\begin{align*}
\int_{\mathcal{S}}|G(1+it)|^2 dt&\ll \left(\frac{TP_1\log X}{X}+1\right)\frac{1}{\log^{2+\varepsilon} X}\\
&+H^2(\log^2 P_j)(\log^{10(j-1)}X)\int_{\mathcal{S}}|Q_{v_j,H,j}(1+it)G_{v_j,H,j}(1+it)|^2 dt
\end{align*}
for some $v_j\in [H\log P_j,(1+\varepsilon)H\log P_j]$ and any measurable $\mathcal{S}\subset [T_0,T]$. In the case $j=1$, this follows from Lemmas \ref{6} and \ref{15}, while in the case $j=2$, we use Perron's formula to separate the variables in $G(s)$. We partition $[T_0,T]$ as $\mathcal{T}_1\cup \mathcal{T}_2\cup \mathcal{T}$ with 
\begin{align*}
\mathcal{T}_1&=\{t\in [T_0,T]:|Q_{v_1,H,1}(1+it)|\leq P_1^{-\alpha_1}\},\\
\mathcal{T}_2&=\{t\in [T_0,T]:|Q_{v_2,H,2}(1+it)|\leq P_2^{-\alpha_2}\}\setminus \mathcal{T}_1,
\end{align*}
and $\mathcal{T}=[T_0,T]\setminus (\mathcal{T}_1\cup \mathcal{T}_2)$.\\

What remains to be done is estimating the integrals
\begin{align*}
B_j=H^2(\log^2 P_j)(\log^{10(j-1)}X)\int_{\mathcal{T}_j}|Q_{v_j,H,j}(1+it)G_{v_j,H,j}(1+it)|^2 dt
\end{align*}
for $j=1,2$, as well as
\begin{align*}
B=H^2(\log^{10} X)\int_{\mathcal{T}}|Q_{v_2,H,2}(1+it)G_{v_2,H,2}(1+it)|^2 dt.
\end{align*}
We have $B_1\ll \left(\frac{TP_1\log X}{X}+1\right)\frac{P_1^{10\varepsilon-\alpha_1}}{\log^2 X}$ by Proposition \ref{p1} and Lemma \ref{15}, and this is small enough since $\alpha_1= 100\varepsilon$. We also have, by Proposition \ref{p2} with $\ell=\lceil \frac{\log P_2}{\log P_1}\rceil$,
\begin{align*}
B_2&\ll H^2(\log^{20} X)P_2^{-2\alpha_2}P_1^{(2+10\varepsilon)\alpha_1\ell}\ell^{(1+o(1))\ell}\\
&\ll P_2^{2(\alpha_1-\alpha_2)+20\varepsilon+\frac{1+2\varepsilon}{a}}\\
&\ll P_2^{-\varepsilon}\ll (\log X)^{-100},
\end{align*}
as long as $a\geq \frac{1}{2(\alpha_2-\alpha_1)}+\frac{C_2}{2}\varepsilon$, say. Lastly, Proposition \ref{p8} gives, for some well-spaced $\mathcal{U}'$ of the type mentioned in the proposition,
\begin{align*}
B\ll (\log X)^{-50}+X^{\frac{1}{2}-\min\{2\sigma(\nu),\frac{\nu}{2}\}+o(1)}\frac{|\mathcal{U}'|X^{\nu+o(1)}}{X}\ll (\log X)^{-50}
\end{align*}
by our assumption on $\mathcal{U}'$. Now the proof is complete.\end{proof}

\subsection{Cases of $\Sigma_1(h)$ and $\Sigma_2(h)$}\label{subsec:Sigma1}

Let $\lambda_d^{+}$ and $\lambda_d^{-}$ be the sieve weights of Brun's pure sieve with $R=2\lfloor (\log \log X)^{\frac{3}{2}}\rfloor$ and sieving parameter $w=\exp(\frac{\log X}{(\log \log X)^3})$. We have
\begin{align*}
\sum_{x\leq p_1dn\leq x+h\atop P_1\leq p_1\leq P_1^{1+\varepsilon} }\lambda_d^{-}\leq \sum_{\substack{\frac{x}{p_1}\leq n\leq \frac{x+h}{p_1}\\P_1\leq p_1\leq P_1^{1+\varepsilon}\\(n,\mathcal{P}(w))=1\\n>1}}1=\Sigma_1(h)\leq  \sum_{x\leq p_1dn\leq x+h\atop P_1\leq p_1\leq P_1^{1+\varepsilon}}\lambda_d^{+}.
\end{align*}
We consider the lower bound; the upper bound can be considered similarly.
Letting $X_1=\frac{X}{T_0^3}$ with $T_0=X^{0.01}$, we have 
\begin{align*}
\frac{h}{X_1}\sum_{P_1\leq p_1\leq P_1^{1+\varepsilon}\atop d\mid \mathcal{P}(w)}\lambda_d^{-}\sum_{\frac{X}{p_1d}\leq n\leq \frac{X+X_1}{p_1 d}}1=h\sum_{P_1\leq p_1\leq P_1^{1+\varepsilon}\atop d\mid \mathcal{P}(w)}\frac{\lambda_d^{-}}{p_1d}+O\left(\frac{h}{X_1}w^{R}P_1^{1+\varepsilon}\right),
\end{align*}
so
\begin{align}\label{eq40}
\Sigma_1(h)&\geq \sum_{d\mid \mathcal{P}(w)\atop P_1\leq p_1\leq P_1^{1+\varepsilon}}\lambda_d^{-}\frac{h}{p_1d}+\left(\sum_{x\leq p_1dn\leq x+h\atop P_1\leq p_1\leq P_1^{1+\varepsilon}}\lambda_d^{-}-\frac{h}{X_1}\sum_{X\leq p_1dn\leq X+X_1\atop P_1\leq p_1\leq P_1^{1+\varepsilon}}\lambda_d^{-}\right)\\
&\quad+O\left(\frac{1}{\log^{100} X}\right).\nonumber
\end{align}
By the fundamental lemma of the sieve (see e.g. \cite[Chapter 6]{friedlander}), we further deduce that
\begin{align*}
\sum_{d\mid \mathcal{P}(w)\atop P_1\leq p_1\leq P_1^{1+\varepsilon}}\lambda_d^{-}\frac{h}{p_1d}&=(1+O((\log X)^{-100}))\sum_{d\mid \mathcal{P}(w)\atop P_1\leq p_1\leq P_1^{1+\varepsilon}}\lambda_d^{+}\frac{h}{p_1d}\\
&\geq \frac{h}{X}\Sigma_1(X)+O\left(\frac{h}{\log^{100} X}\right).
\end{align*}
Therefore, we may concentrate on the expression in the parentheses in \eqref{eq40}, which is a difference between a short and long average.  By Lemma \ref{1}, it is $o\left(\frac{h}{\log X}\right)$ for $h\geq P_1\log X$ and for almost all $x\leq X$, provided that
\begin{align*}
\int_{T_0}^{T}|F(1+it)|^2dt=o\left(\left(\frac{TP_1\log X}{X}+1\right)\frac{1}{\log^2 X}\right)
\end{align*}
for all $T\geq T_0$, where $T_0=X^{0.01}$, and
\begin{align*}
F(s)=\sum_{p_1dn\sim X\atop P_1\leq p_1\leq P_1^{1+\varepsilon}}\lambda_d^{-}(p_1dn)^{-s}.
\end{align*}

Such an estimate is given by the following proposition, which is invoked again in the case of the sum $\Sigma_2(h)$.

\begin{proposition}\label{p7}Let  $\varepsilon>0$, $P_1=\log^{a} X$ with $a\geq 2+C_3\varepsilon$ and
\begin{align*}
F(s)=\sum_{p_1dn\sim X\atop P_1\leq p_1\leq P_1^{1+\varepsilon}}\lambda_d^{\pm}(p_1dn)^{-s}\quad \text{or}\quad F(s)=\sum_{\substack{p_1pdn\sim X\\ P_1\leq p_1\leq P_1^{1+\varepsilon}\\M\leq p\leq M^{1+\varepsilon}}}\lambda_d^{\pm}(p_1pdn)^{-s}
\end{align*}
with $M\ll X^{\frac{1}{2}+o(1)}$ , $X^{1+o(1)}\gg T\geq T_0=X^{0.01}$ as before, and either $+$ or $-$ sign chosen throughout. Then,
\begin{align*}
\int_{T_0}^{T}|F(1+it)|^2 dt\ll \left(\frac{TP_1\log X}{X}+1\right)\frac{1}{\log^{2+\varepsilon} X}.
\end{align*}
\end{proposition}

\begin{proof} Let $D$ be a large constant, and for positive integer $v$ and $H=\log^{10\varepsilon} X$ denote
\begin{align*}
P_{v,H}(s)=\sum_{e^{\frac{v}{H}}\leq p<e^{\frac{v+1}{H}}}p^{-s}
\end{align*}
and
\begin{align*}
F_{v,H}(s)=\sum_{dn\sim Xe^{-\frac{v}{H}}}\lambda_{d}^{\pm}(dn)^{-s}\quad \text{or}\quad F_{v,H}(s)=\sum_{pdn\sim Xe^{-\frac{v}{H}}\atop M\leq p\leq M^{1+\varepsilon}}\lambda_{d}^{\pm}(pdn)^{-s}.
\end{align*}
Lemma \ref{6} gives
\begin{align}\label{eq11}
\int_{T_0}^{T}|F(1+it)|^2dt&\ll H^2(\log \log X)^2\int_{T_0}^{T}|P_{v_0,H}(1+it)F_{v_0,H}(1+it)|^2dt\nonumber\\
&+T\sum_{n\in [Xe^{-\frac{1}{H}},Xe^{\frac{1}{H}}]\, or\atop n\in [2X,2Xe^{\frac{1}{H}}]}|a_n|^2+T\sum_{1\leq h\leq \frac{X}{T}}\sum_{\substack{m-n=h\\ m,n\in [Xe^{-\frac{1}{H}},Xe^{\frac{1}{H}}]\,or\\ m,n\in [2X,2Xe^{\frac{1}{H}}]}}|a_m||a_n|,
\end{align}
for some $v_0\in I_0$, where $I_0=[H\log P_1,H\log P_1^{1+\varepsilon}]$ and 
\begin{align}\label{eq2}
a_m=\sum_{p_1\mid m\atop P_1\leq p_1\leq P_1^{1+\varepsilon}}\left|\sum_{m=p_1dn}\lambda_d^{\pm}\right|\quad \text{or}\quad a_m=\sum_{p_1\mid m\atop P_1\leq p_1\leq P_1^{1+\varepsilon}}\left|\sum_{m=p_1pdn\atop M\leq p\leq M^{1+\varepsilon}}\lambda_d^{\pm}\right|
\end{align}

Lemma \ref{5} tells that the last two terms in \eqref{eq11} contribute, for some constant $C>0$, 
\begin{align*}
&\ll \frac{T}{X}\left(\frac{(\log \log X)^{C}}{H}\cdot \frac{1}{\log X}+\frac{(\log \log X)^C}{H}\cdot \frac{X}{T}\cdot \frac{1}{\log^2 X}\right)\nonumber\\
&\ll \left(\frac{TP_1\log X}{X}+1\right)\cdot\frac{1}{\log^{2+\varepsilon} X}
\end{align*}
by the definition of $H$. We are now left with estimating the integral in \eqref{eq11}. We consider the integrals in two parts, namely the part over $\mathcal{T}_1$ and its complement, with 
\begin{align*}
\mathcal{T}_1=\{t\in [T_0,T]:|P_{v_0,H}(1+it)|\leq P_1^{-100\varepsilon}\}.
\end{align*}

The case of $\mathcal{T}_1$ is dealt with Proposition \ref{p1} and Lemma \ref{5}, and it contributes 
\begin{align*}
&\ll H^2(\log \log X)^2\frac{T}{X}P_1^{1-200\varepsilon}\left(S_1\left(\frac{X}{P_1},(a_n)\right)+S_2\left(\frac{X}{P_1},(a_n)\right)\right)\\
&\ll (\log \log X)^C\left(\frac{T}{X}\cdot \frac{1}{\log X}+\frac{1}{P_1}\cdot \frac{1}{\log^2 X}\right)\cdot P_1^{1-100\varepsilon}\\
&\ll \left(\frac{TP_1\log X}{X}+1\right)\cdot \frac{1}{\log^{2+\varepsilon} X},
\end{align*}
where the coefficients $a_n$ involved in definition of $S_i(X,(a_n))$ are given by \eqref{eq2}.\\

We turn to the integral over the complement of $\mathcal{T}_1$ and resort to the Watt-type Proposition \ref{p5}. Let $\ell$ be a large positive integer such that $P_1^{\ell}= X^{\varepsilon+o(1)}$. Letting $N_a(s)=\sum_{n\sim Xe^{-a}}n^{-s}$ and
\begin{align*}
M_{v,H}(s)=\sum_{\substack{e^{\frac{v}{H}}\leq p_1d<e^{\frac{v+1}{H}}\\P_1\leq p_1\leq P_1^{1+\varepsilon}}}\lambda_{d}^{\pm}(p_1d)^{-s}\quad \text{or}\quad M_{v,H}(s)=\sum_{\substack{e^{\frac{v}{H}}\leq p_1pd<e^{\frac{v+1}{H}}\\P_1\leq p_1\leq P_1^{1+\varepsilon}\\M\leq p\leq M^{1+\varepsilon}}}\lambda_{d}^{\pm}(p_1pd)^{-s},
\end{align*}
an application of Perron's formula to separate variables, along with Lemma \ref{5} and $|P_{v_1,H}(1+it)P_1^{100\varepsilon}|^{2\ell}\geq 1$, yields
\begin{align}\label{eq9}
&\int_{[T_0,T]\setminus \mathcal{T}_1}|F(1+it)|^2 dt\nonumber\\
&\ll H^2(\log^{10} X)P_1^{200\varepsilon \ell} \int_{T_0}^{T}|P_{v_0,H}(1+it)^{\ell}M_{v_1,H}(1+it)N_{\frac{v_1}{H}}(1+it)|^2 dt\\
&+\left(\frac{TP_1\log X}{X}+1\right)\cdot\frac{1}{\log^{2+\varepsilon} X}\nonumber
\end{align}
for some $v_1\in I_1,$ where $I_1=[H\log M,H\log (M^{1+\varepsilon}w^R)].$ Now Proposition \ref{p5} with $N(s)=N_{\frac{v_1}{H}}(s),$ $M(s)=M_{v_1,H}(s),$ $P(s)\equiv 1,Q(s)=P_{v_0,H}(s)^{\ell}$ and $\ell=\lfloor \frac{\varepsilon \log X}{\log P_1}\rfloor$ bounds \eqref{eq9} with
\begin{align}\label{eq41}
X^{o(1)}P_1^{200\varepsilon \ell}\left(Q^{-1}+\frac{1}{T_0}\right)(\ell!)^2\ll (P_1^{-1}(\log^2 X))^{(1+o(1))\ell}+X^{-\varepsilon}\ll X^{-\varepsilon^2}
\end{align}
for $a\geq 2+C_3\varepsilon$, since the condition $M^2P\ll X^{1+o(1)}$ certainly holds. \end{proof}

Note that Proposition \ref{p7} immediately shows that 
\begin{align*}
\frac{1}{h}\Sigma_1(h)-\frac{1}{X_1}\Sigma_1(X_1)\geq o\left(\frac{1}{\log X}\right)
\end{align*}
for almost all $x\leq X$, where $X_1=\frac{X}{T_0^3}$. Taking into account formula \eqref{eq40} and repeating the above argument with lower bound sieve weights replaced with upper bound sieve weights, we see that the reverse inequality holds,
so $\frac{1}{h}\Sigma_1(h)$ can be replaced with its dyadic counterpart $\frac{1}{X}\Sigma_1(X)$ almost always.\\

 Now we deal with $\Sigma_2(h)$. We use the same strategy, so that for example for the lower bound we start with
\begin{align*}
\Sigma_2 \geq \sum_{\substack{x\leq p_1pdn\leq x+h\\P_1\leq p_1\leq P_1^{1+\varepsilon}}}\lambda_d^{-},
\end{align*}
an inequality that is valid even when the interval $\left[\frac{x}{p_1p},\frac{x+h}{p_1p}\right]$ contains no integers. This leads us to study the Dirichlet polynomial
\begin{align*}
F^{*}(s)=\sum_{\substack{p_1pdn\sim X\\P_1\leq p_1\leq P_1^{1+\varepsilon}\\w\leq p<\sqrt{x}}}\lambda_d^{-}(p_1pdn)^{-s},
\end{align*}
where the variable $p$ can be divided into $\ll \log \log X$ intervals of the form $[M,M^{1+\varepsilon}]$ with $M\ll X^{\frac{1}{2}+o(1)}$ (the value of $\varepsilon$ may be varied so that the division becomes exact). For each of these Dirichlet polynomials where $p$ is restricted, Proposition \ref{p7} gives a bound of $\left(\frac{TP_1\log X}{X}+1\right)(\log X)^{-2-\varepsilon}$ for their second moment. Now by the same argument as for $\Sigma_1(h)$, we infer that $\frac{1}{h}\Sigma_2(h)$ can also be replaced with its dyadic counterpart $\frac{1}{X}\Sigma_2(X)$ almost always. 

\subsection{Case of $\Sigma_3(h)$}\label{subsec:sigma}

We are left with the sum $\Sigma_3(h)$. This is the case that determines which value of $a$ we obtain (and hence the value of $c$, which is just $a+1$), since so far in all cases $a\geq 2+C_4\varepsilon$ has been a sufficient assumption. We will establish the value $a=2.51$.\\ 

Let $\beta_1,\beta_2,\beta\in (\frac{1}{6},\frac{1}{2})$ be parameters which are given the values
\begin{align*}
\beta_1=0.1680,\quad \beta_2=0.1803,\quad \beta=0.1950
\end{align*}
to optimize various subsequent conditions. We split $\Sigma_3(h)$ into three parts $\Sigma_3^{(1)}(h),\Sigma_3^{(2)}(h)$ and $\Sigma_3^{(3)}(h)$, say, the first sum being a type II sum that can be evaluated asymptotically, the second being a type I sum (after Buchstab's identity) that can mostly be evaluated asymptotically, and the third being a type II sum that can be transformed into Buchstab integrals whose value is suitably small. Explicitly, let
\begin{align*}
\Sigma_3^{(i)}(h)=\sum_{\substack{x\leq p_1q_1q_2n\leq x+h\\P_1\leq p_1\leq P_1^{1+\varepsilon}\\(q_1,q_2)\in A_i\\(n,\mathcal{P}(q_2))=1\\n>1}}1,\quad i=1,2,3
\end{align*}
with
\begin{align*}
A_1=&\{(q_1,q_2):\,\,w\leq q_2<q_1,\,\, \text{one of}\,\, q_1,q_2\in [w,X^{\beta_1}]\cup [X^{\beta_2},X^{\beta}]\},\\
A_2=&\{(q_1,q_2):\,\,w\leq q_2<q_1,\,\, \text{either}\,\, q_1^2q_2^3\leq X\,\,\text{or}\,\, q_1q_2^4\leq X,\,\,q_1\leq X^{\frac{1}{4}-2\varepsilon}\}\setminus A_1\\
A_3=&\{(q_1,q_2):\,\,w\leq q_2<q_1\leq X^{\frac{1}{2}}\}\setminus (A_1\cup A_2).
\end{align*}

The underlying idea is that the small variable in $A_1$ enables efficient use of large values theorems, the conditions in $A_2$ make it possible to apply Watt's theorem (after two applications of Buchstab's identity), and the remaining set $A_3$ can be shown to contribute not too much. We study the sums $\Sigma_3^{(i)}(h)$ separately, starting with $\Sigma_3^{(1)}(h)$.

\subsubsection{Type II sums}

We consider the Type II sum $\Sigma_3^{(1)}(h)$. In order to prove that $\frac{1}{h}\Sigma_3^{(1)}(h)$ is asymptotically $\frac{1}{X}\Sigma_3^{(1)}(X)$ almost always, it suffices to prove that $\frac{1}{h}\Sigma_3^{(1)}(h)$ is asymptotically $\frac{1}{X_1}\Sigma_3^{(1)}(X_1)$ almost always with $X_1=\frac{X}{T_0^3}$, and then apply the prime number theorem in short intervals. For this latter asymptotic equivalence, it suffices to show that the Dirichlet polynomial
\begin{align*}
G(s)=\sum_{\substack{p_1q_1q_2n\sim X\\P_1\leq p_1\leq P_1^{1+\varepsilon}\\Q_i\leq q_i\leq P_i^{1+\varepsilon}, i\leq 2\\q_2<q_1\\(n,\mathcal{P}(q_2))=1\\n>1}}(p_1q_1q_2n)^{-s}
\end{align*}
satisfies
\begin{align*}
\int_{T_0}^T |G(1+it)|^2 dt\ll \left(\frac{TP_1\log X}{X}+1\right)\frac{1}{\log^{2+\varepsilon} X}
\end{align*}
with $T\leq X^{1+o(1)}$, $T_0=X^{0.01}$, $P_1=\log^a X$ and $Q_1,Q_2\geq w$ otherwise arbitrary, but either $Q_1$ or $Q_2$ is of size $X^{\nu+o(1)}$ with $\nu \in [0,\beta_1]\cup [\beta_2,\beta]$. These cases are similar, so assume $Q_2=X^{\nu+o(1)}$ with $\nu$ as above.\\

This is the setting of Proposition \ref{p6}. Therefore, if for every polynomial of the form
\begin{align*}
M(s)=\sum_{m\sim M}\frac{b_m}{m^s},
\end{align*}
with $M=X^{\nu+o(1)}$ and $|b_m|\leq d_r(n)$ for fixed $r$, any well-spaced set
\begin{align*}
\mathcal{U}'\subset \{t\in [0,T]:\,\, |M(1+it)|\geq M^{-\alpha_2}\}
\end{align*}
satisfies
\begin{align*}
|\mathcal{U}'|\ll X^{\frac{1}{2}-\nu+\min\{2\sigma(\nu),\frac{\nu}{2}\}-\varepsilon},
\end{align*}
the sum $\Sigma_3^{(1)}(h)$ has the anticipated asymptotic for $a\geq \frac{1}{2\alpha_2}+C_5\varepsilon$. Of course, we fix $\alpha_2=\frac{1}{2\cdot 2.51}+C_6\varepsilon$.\\

We are left with estimating $|\mathcal{U}'|$, and to this end we utilize Jutila's large values theorem. Jutila's large values theorem (Lemma \ref{13}) applied to the $\ell$th moment of $M(s)$ can be reformulated to say that if
\begin{align*}
\mathcal{R}(\nu,\alpha_2,k,\ell)=\max\left\{2\nu \alpha_2 \ell,\left(6-\frac{2}{k}\right)\nu \alpha_2 \ell+1-2\nu \ell,\,\,1+8k\ell\nu \alpha_2-2k\ell \nu\right\}
\end{align*}
and
\begin{align*}
\overline{\mathcal{R}}(\nu,\alpha_2)=\min_{k,\ell\in \{1,2,...\}}\mathcal{R}(\nu,\alpha_2,k,\ell),
\end{align*}
then $|\mathcal{U}|\ll X^{\tilde{\mathcal{R}}(\nu,\alpha_2)+o(1)}$. It turns out that the case $k=3$ is always optimal for us, and it suffices to restrict to $4\leq \ell\leq 12$ (so our upper bound for $\overline{\mathcal{R}}(\nu,\alpha_2)$ is a minimum of $9$ piecewise linear functions). Now we check that, with our choices of $\beta_1, \beta_2, \beta$ and $\alpha_2$,
\begin{align*}
\overline{\mathcal{R}}(\nu,\alpha_2)\leq \frac{1}{2}-\nu+\min\left\{2\sigma(\nu),\frac{\nu}{2}\right\}-\varepsilon\end{align*}
for $\nu \in [0.05,\beta_1]\cup [\beta_1,\beta_2]$. Verifying this is straightforward, because both sides are piecewise linear functions.\footnote{These computations can be carried out by hand with a bit of patience. For example, the case $\ell=4$ in Jutila's bound is good enough in the range $\nu \in [\frac{16315}{90496},\frac{15311}{78512}]$, and the bound for $\ell=5$ is good enough when $\nu \in [\frac{753}{5554},\frac{15311}{91112}]$. These intervals are $[\beta_2,\beta]$ and $[0.1356,\beta_1]$, up to rounding.}\\

We must also prove the desired estimate for $|\mathcal{U}'|$ in the range $\nu\in [0,0.05).$ In this case, we do not appeal to Jutila's large values theorem, but to Lemma \ref{7} (along with its remark), which tells us that
\begin{align*}
|\mathcal{U}'|\ll T^{2\alpha_2}X^{2\alpha_2\nu+o(1)}\ll X^{0.42}<X^{\frac{1}{2}-\nu-\varepsilon}
\end{align*}
for the same value $\alpha_2=\frac{1}{2\cdot 2.51}+C_6\varepsilon$. This means that for $c=3.51$, $\frac{1}{h}\Sigma_3^{(1)}(h)$ can be replaced with its dyadic counterpart almost always.

\subsubsection{Type I sums}

We turn to the sum $\Sigma_3^{(2)}(h)$. By applying Buchstab's identity twice, we find that
\begin{align*}
\Sigma_3^{(2)}(h)&=\sum_{\substack{x\leq p_1q_1q_2n\leq x+h\\P_1\leq p_1\leq P_1^{1+\varepsilon}\\(q_1,q_2)\in A_2\\(n,\mathcal{P}(w))=1\\n>1}}1-\sum_{\substack{x\leq p_1q_1q_2q_3n\leq x+h\\P_1\leq p_1\leq P_1^{1+\varepsilon}\\(q_1,q_2)\in A_2\\w\leq q_3<q_2\\(n,\mathcal{P}(w))=1\\n>1}}1+\sum_{\substack{x\leq p_1q_1q_2q_3q_4n\leq x+h\\P_1\leq p_1\leq P_1^{1+\varepsilon}\\(q_1,q_2)\in A_2\\w\leq q_4<q_3<q_2\\(n,\mathcal{P}(q_4))=1\\n>1}}1.
\end{align*}
Call these sums $\Sigma_3^{(2,1)}(h),\Sigma_3^{(2,2)}(h)$ and $\Sigma_3^{(2,3)}(h)$, respectively. We show that $\frac{1}{h}\Sigma_3^{(2,1)}(h)$ and $\frac{1}{h}\Sigma_3^{(2,2)}(h)$ can be replaced with their dyadic counterparts almost always. We confine to studying $\Sigma_3^{(2,2)}(h)$, as $\Sigma_3^{(2,1)}(h)$ is easier to handle.\\

We may make in $\Sigma_3^{(2,2)}(h)$ the additional assumption that all the variables except $P_1$ are in the intervals $[X^{\beta_1},X^{\beta_2}]\cup[X^{\beta},X]$, since otherwise the sum can be dealt with in the same way as $\Sigma_3^{(1)}(h)$. We may also assume that $q_i\in [Q_i,Q_i^{1+\varepsilon}]$ for some $Q_i$. Defining
\begin{align*}
F(s)=\sum_{\substack{p_1q_1q_2q_3dn\sim X\\P_1\leq p_1\leq P_1^{1+\varepsilon}\\Q_i\leq q_i\leq Q_i^{1+\varepsilon}\\(q_1,q_2)\in A}}\lambda_d^{\pm} (p_1q_1q_2q_3dn)^{-s},
\end{align*}
with $\lambda_d^{\pm}$ the same Brun's sieve weights as before (the sign being the same throughout), and taking into account the prime number theorem in short intervals and Lemmas \ref{1} and \ref{6}, it suffices to show that
\begin{align*}
\int_{T_0}^{T}|F(1+it)|^2 dt\ll \left(\frac{TP_1\log X}{X}+1\right)\frac{1}{\log^{2+\varepsilon} X}.
\end{align*}
This bound is achieved similarly as in Proposition \ref{p7}. Indeed, if $\mathcal{T}_1$  is defined as in the proof of that proposition, the integral over $\mathcal{T}_1$ can be estimated in the same way as in that proposition. In the complementary case, we separate all the variables, and it remains to show that
\begin{align*}
&\int_{[T_0,T]\setminus \mathcal{T}_1} |P_1(1+it)Q_1(1+it)Q_2(1+it)Q_3(1+it)D(1+it)N(1+it)|^2 dt\\
&\ll (\log X)^{-100},
\end{align*}
where $N(s)$ is a zeta sum, $P_1(s)$ and $Q_i(s)$ are polynomials supported on primes, and $D(s)$ has the sieve weights $\lambda_d$ as its coefficients (actually, $D(s)$ can be neglected by simply estimating it pointwise). Moreover, the lengths $P_1,Q_i, D$ and $N$ are from the same intervals as $p_1,q_i,d$ and $n$, respectively (in particular, $d\leq \exp\left(\frac{\log X}{\log \log X}\right)$). We appeal to Proposition \ref{p5} with $Q(s)=P_1(s)^{\ell}$, $P_1^{\ell}=X^{\varepsilon}$ and with $M(s)$ either $Q_1(s)Q_3(s)$ or $Q_2(s)Q_3(s)$. If $M(s)=Q_1(s)Q_3(s)$, the condition for Proposition $\ref{p5}$ is $Q_2\leq X^{\frac{1}{4}-2\varepsilon}$, $(Q_1Q_3)^2Q_2\leq X$. If in turn $M(s)=Q_2(s)Q_3(s)$, the condition for Proposition \ref{p5} is $Q_1\leq X^{\frac{1}{4}-2\varepsilon}$, $Q_1(Q_2Q_3)^2\leq X$, and one of these conditions is always satisfied in our domain $A_2$, since $Q_3\leq Q_2$ and automatically $Q_2\leq X^{\frac{1}{5}}$. Now it follows from \eqref{eq41} that for $a\geq 2+C_7\varepsilon$, $\Sigma_3^{(2,2)}(h)$ has the desired asymptotic, and $\Sigma_3^{(2,1)}(h)$ can be evaluated similarly.\\

In the sum $\Sigma_3^{(2,3)}(h)$, we may again assume that all the variables lie in the intervals $[X^{\beta_1},X^{\beta_2}]\cup[X^{\beta},X]$, as otherwise we can use the type II sum argument. Let $\Sigma_3^{(2,4)}(h)$ be what remains of $\Sigma_3^{(2,3)}(h)$ after this reduction. The sum $\Sigma_3^{(2,4)}$ results in a Buchstab integral, and hence is postponed to Subsection \ref{subsubsec:buchstab}.

\subsubsection{Buchstab integrals}\label{subsubsec:buchstab}

We are left with the sums $\Sigma_3^{(3)}(h)$ and $\Sigma_3^{(2,4)}(h)$, for which no asymptotic was found. We want to show that 
\begin{align*}
\frac{1}{X}\Sigma_3^{(3)}(X)+\frac{1}{X}\Sigma_3^{(2,4)}(X)\leq (1-\varepsilon)\frac{1}{X}S_X,
\end{align*}
which would complete the proof of Theorem \ref{t5}, taking into account the estimates \eqref{eq38} and \eqref{eq39}. The following lemma allows us to transform our sums into Buchstab integrals.

\begin{lemma}
Let a positive integer $k$ and $\eta>0$ be fixed. Let 
\begin{align*}
A\subset \{(u_1,...,u_k)\in \mathbb{R}^k:\,\, u_1,...,u_k\geq \eta,\,\, u_1+...+u_k\leq 1-\eta\}
\end{align*}
be any set such that $1_A$ is Riemann integrable. For a point $q=(q_1,...,q_k)\in \mathbb{R}^k$ and $X\geq 2$, define $\mathcal{L}(q)=(\frac{\log q_1}{\log X},...,\frac{\log q_k}{\log X})$. Then
\begin{align*}
&\sum_{\substack{p_1q_1\dotsm q_k n\sim X\\P_1\leq p_1\leq P_1^{1+\varepsilon}\\\mathcal{L}(q_1,...,q_k)\in A\\(n,\mathcal{P}(q_k))=1}}1\\
&=(1+o(1))\log(1+\varepsilon)\frac{X}{\log X}\int_{(u_1,...,u_k)\in A}\omega\left(\frac{1-u_1-\dotsm -u_k}{u_k}\right)\frac{du}{u_1\dotsm u_{k-1}u_k^2},
\end{align*}
where $\omega(\cdot)$ is Buchstab's function.
\end{lemma}

\begin{proof} It suffices to prove the statement in the case that $A$ is a box, that is, a set of the form $I_1\times ...\times I_k$ with $I_i$ intervals. Indeed, if the statement holds for boxes, then it holds for finite unions of boxes. Moreover, since $1_A$ is Riemann integrable, for every $\delta>0$ there is a finite union $\mathcal{B}$ of boxes such that $A\setminus \mathcal{B}$ has measure at most $\delta$. The part of $A$ not contained in $\mathcal{B}$ contributes at most $\eta^{-k-1}\delta$ to the integral, and as $\delta\to 0$, this becomes arbitrarily small.\\

Now let $A$ be a box. Using the connection between Buchstab's function and the sieving function (see the Appendix of Harman's book \cite{harman-sieves}), summing partially, and using the change of variables $u_i=\frac{\log v_i}{\log X}$, we see that
\begin{align*}
\sum_{\substack{p_1q_1\dotsm q_k n\sim X\\P_1\leq p_1\leq P_1^{1+\varepsilon}\\\mathcal{L}(q_1,...,q_k)\in A\\(n,\mathcal{P}(q_k))=1}}1&=\sum_{\substack{P_1\leq p_1\leq P_1^{1+\varepsilon}\\\mathcal{L}(q_1,...,q_k)\in A}}S\left(\left[\frac{X}{p_1q_1\dotsm q_k},\frac{2X}{p_1q_1\dotsm q_k}\right],\mathbb{P},q_k\right)\\
&=(1+o(1))\sum_{P_1\leq p_1\leq P_1^{1+\varepsilon}\atop \mathcal{L}(q_1,...,q_k)\in A}\frac{X}{p_1q_1\dotsm q_k\log q_k}\omega\left(\frac{\log \frac{X}{p_1q_1\dotsm q_k}}{\log q_k}\right)\\
&=(1+o(1))\sum_{P_1\leq p_1\leq P_1^{1+\varepsilon}}\frac{1}{p_1}\sum_{\mathcal{L}(q_1,...,q_k)\in A}\frac{X}{q_1...q_k\log q_k}\omega\left(\frac{\log \frac{X}{q_1...q_k}}{\log q_k}\right)\\
&=(b+o(1))\int\limits_{\mathcal{L}(v_1,...,v_k)\in A}\frac{X}{v_1\dotsm v_k\log v_1\dotsm \log^2 v_k}\omega\left(\frac{\log \frac{X}{v_1\dotsm v_k}}{\log v_k}\right)dv\\
&=(b+o(1))\frac{X}{\log X}\int\limits_{(u_1,...,u_k)\in A}\frac{1}{u_1\dotsm u_k^2}\omega\left(\frac{ 1-u_1-\dotsm u_k}{u_k}\right)du
\end{align*}
with $b=\log(1+\varepsilon)$, as wanted.\end{proof}

Let 
\begin{align*}
A_3^{*}=&\{(u_1,u_2):\,\, u_2<u_1,\,\,u_1,u_2\in [\beta_1,\beta_2]\cup [\beta,\frac{1}{2}],\,\,2u_1+3u_2\geq 1,\\
&\,\,\max\{u_1+4u_2,4u_1-10\varepsilon\}\geq 1\},\\
A_2^{*}=&\{(u_1,u_2,u_3,u_4):\,\, \beta_1\leq u_4<u_3<u_2<u_1,\,\,u_1,u_2,u_3,u_4\not\in [\beta_2,\beta],\,\, (u_1,u_2)\in A_2\}
\end{align*}
be the sets corresponding to the summation conditions in $\Sigma_3^{(3)}(X)$ and $\Sigma_3^{(2,4)}(X)$, respectively. The lemma above directly implies that
\begin{align*}
\frac{1}{X}\Sigma_3^{(3)}(X)&=\frac{(1+o(1))\log(1+\varepsilon)}{\log X}J_1,\\
\frac{1}{X}\Sigma_3^{(2,4)}(X)&=\frac{(1+o(1))\log(1+\varepsilon)}{\log X}J_2,\\
\frac{1}{X}S_X&=\frac{(1+o(1))\log(1+\varepsilon)}{\log X}
\end{align*}
where $J_1$ and $J_2$ are given by
\begin{align*}
J_1&=\int\limits_{(u_1,u_2)\in A_3^{*}}\omega\left(\frac{1-u_1-u_2}{u_2}\right)\frac{du}{u_1u_2^2},\\
J_2&=\int\limits_{(u_1,u_2,u_3,u_4)\in A_2^{*}}\omega\left(\frac{1-u_1-u_2-u_3-u_4}{u_4}\right)\frac{du}{u_1u_2u_3u_4^2}.
\end{align*}
To compute $J_1$, we approximate Buchstab's function by
\begin{align*}
\omega(u)\leq \begin{cases}0,\quad u<1\\\frac{1}{u},\quad 1\leq u\leq 2\\\frac{1+\log(u-1)}{u},\quad 2\leq u\leq 3\\\frac{1+\log 2}{3},\quad u>3\end{cases}
\end{align*}
For $u\leq 3$ this is an equality, and for $u>3$ the bound very sharp (it differs from the limiting value $e^{-\gamma}$, where $\gamma$ is Euler's constant, by less than $0.003$), but we only need the fact that it is an upper bound. We compute with Mathematica that $J_1<0.988$ (when $\varepsilon$ in the definition of $A_3^{*}$ is small enough).\footnote{The Mathematica code can be found at \nolinkurl{http://codepad.org/XCqx2iH3} . There is also a Python code for computing the integral at \nolinkurl{http://codepad.org/cVx065z5}, where the integration method is a rigorous computation of an upper Riemann sum.}\,\, The integral $J_2$ only gives a minor contribution, and hence can be estimated crudely as
\begin{align*}
J_2&\leq \beta_1^{-5}\int\limits_{(u_1,u_2,u_3,u_4)\in A_2^{*}\atop u_1+u_2+u_3+2u_4\leq 1}du\\
&<\beta_1^{-5}\int\limits_{\substack{\beta_1<u_4<u_3<u_2<u_1\\u_1+u_2+u_3+2u_4\leq 1}}du<0.007
\end{align*}
with Mathematica (the last integral could actually be evaluated exactly). To sum up, we have $J_1+J_2<0.995<1-\varepsilon$, and this means, in view of \eqref{eq39}, that with our parameter choices $\beta_1,\beta_2,\beta$, the sums $\Sigma_3^{(3)}(X)$ and $\Sigma_3^{(2,4)}(X)$ can be discarded. Now, from \eqref{eq38} and \eqref{eq39} we have $\frac{1}{h}S_h(x)\geq \varepsilon\cdot \frac{1}{X}S_X$, so Theorem \ref{t5} is proved.\qedd

\begin{remark}
We can now observe that $c=3+\varepsilon$ is the limit of this method. Indeed, we are forced to take $\alpha_2\leq \frac{1}{4}$ in the type II case, because nothing nontrivial is known about the large values of Dirichlet polynomials beyond this region, and consequently $a=\frac{1}{2\alpha_2}+\varepsilon\geq 2+\varepsilon$ and $c\geq 3+\varepsilon$.
\end{remark}

\bibliography{myreferences}{}

\begin{thebibliography}{10}

\bibitem{baker-harman-pintz}
R.~C. Baker, G.~Harman, and J.~Pintz.
\newblock The difference between consecutive primes. {II}.
\newblock {\em Proc. London Math. Soc. (3)}, 83(3):532--562, 2001.

\bibitem{bourgain}
J.~Bourgain.
\newblock On large values estimates for {D}irichlet polynomials and the density
  hypothesis for the {R}iemann zeta function.
\newblock {\em Internat. Math. Res. Notices}, (3):133--146, 2000.

\bibitem{freiberg}
T.~Freiberg.
\newblock Short intervals with a given number of primes.
\newblock {\em J. Number Theory}, 163:159--171, 2016.

\bibitem{friedlander}
J.~Friedlander and H.~Iwaniec.
\newblock {\em Opera de cribro}, volume~57 of {\em American Mathematical
  Society Colloquium Publications}.
\newblock American Mathematical Society, Providence, RI, 2010.

\bibitem{gallagher}
P.~X. Gallagher.
\newblock On the distribution of primes in short intervals.
\newblock {\em Mathematika}, 23(1):4--9, 1976.

\bibitem{goldston1}
D.~A. Goldston, J.~Pintz, and C.~Y. Y{\i}ld{\i}r{\i}m.
\newblock Positive proportion of small gaps between consecutive primes.
\newblock {\em Publ. Math. Debrecen}, 79(3-4):433--444, 2011.

\bibitem{goldston2}
D.~A. Goldston, J.~Pintz, and C.~Y. Y{\i}ld{\i}r{\i}m.
\newblock Primes in tuples {IV}: {D}ensity of small gaps between consecutive
  primes.
\newblock {\em Acta Arith.}, 160(1):37--53, 2013.

\bibitem{hardy}
G.~H. Hardy and S.~Ramanujan.
\newblock The normal number of prime factors of a number {$n$} [{Q}uart. {J}.
  {M}ath. {\bf 48} (1917), 76--92].
\newblock In {\em Collected papers of {S}rinivasa {R}amanujan}, pages 262--275.
  AMS Chelsea Publ., Providence, RI, 2000.

\bibitem{harman-almostprimes}
G.~Harman.
\newblock Almost-primes in short intervals.
\newblock {\em Math. Ann.}, 258(1):107--112, 1981/82.

\bibitem{harman-sieves}
G.~Harman.
\newblock {\em Prime-detecting sieves}, volume~33 of {\em London Mathematical
  Society Monographs Series}.
\newblock Princeton University Press, Princeton, NJ, 2007.

\bibitem{heath-brown-vaughan}
D.~R. Heath-Brown.
\newblock Prime numbers in short intervals and a generalized {V}aughan
  identity.
\newblock {\em Canad. J. Math.}, 34(6):1365--1377, 1982.

\bibitem{iwaniec-kowalski}
H.~Iwaniec and E.~Kowalski.
\newblock {\em Analytic number theory}, volume~53 of {\em American Mathematical
  Society Colloquium Publications}.
\newblock American Mathematical Society, Providence, RI, 2004.

\bibitem{jia}
C.~Jia.
\newblock Almost all short intervals containing prime numbers.
\newblock {\em Acta Arith.}, 76(1):21--84, 1996.

\bibitem{jutila-density}
M.~Jutila.
\newblock Zero-density estimates for {$L$}-functions.
\newblock {\em Acta Arith.}, 32(1):55--62, 1977.

\bibitem{matomaki}
K.~{Matom{\"a}ki} and M.~{Radziwi{\l}{\l}}.
\newblock {Multiplicative functions in short intervals}.
\newblock {\em \textnormal{To appear in} Ann. of Math.}

\bibitem{mikawa}
H.~Mikawa.
\newblock Almost-primes in arithmetic progressions and short intervals.
\newblock {\em Tsukuba J. Math.}, 13(2):387--401, 1989.

\bibitem{montgomery}
H.~L. Montgomery.
\newblock {\em Ten lectures on the interface between analytic number theory and
  harmonic analysis}, volume~84 of {\em CBMS Regional Conference Series in
  Mathematics}.
\newblock Published for the Conference Board of the Mathematical Sciences,
  Washington, DC; by the American Mathematical Society, Providence, RI, 1994.

\bibitem{selberg}
A.~Selberg.
\newblock On the normal density of primes in small intervals, and the
  difference between consecutive primes.
\newblock {\em Arch. Math. Naturvid.}, 47(6):87--105, 1943.

\bibitem{watt-thm}
N.~Watt.
\newblock Kloosterman sums and a mean value for {D}irichlet polynomials.
\newblock {\em J. Number Theory}, 53(1):179--210, 1995.

\bibitem{watt-primes}
N.~Watt.
\newblock Short intervals almost all containing primes.
\newblock {\em Acta Arith.}, 72(2):131--167, 1995.

\bibitem{wolke}
D.~Wolke.
\newblock Fast-{P}rimzahlen in kurzen {I}ntervallen.
\newblock {\em Math. Ann.}, 244(3):233--242, 1979.

\end{thebibliography}
\bibliographystyle{plain}

\textsc{Department of Mathematics and statistics, University of Turku, 20014 Turku, Finland}\\
\textit{Email address:} \textup{\texttt{joni.p.teravainen@utu.fi}}

\end{document}